\documentclass[11pt,reqno]{amsart}

\usepackage[T1]{fontenc}
\usepackage{lmodern}
\usepackage{microtype}
\usepackage{amssymb,mathrsfs}
\usepackage{enumitem}
\usepackage{hyperref}
\usepackage[nameinlink,capitalize,noabbrev]{cleveref}

\hypersetup{
  hidelinks,
  pdftitle={Realizing prescribed entropy functions by smooth diffeomorphisms of closed manifolds of dimension at least three},
  pdfauthor={Wanshan Lin and Xueting Tian},
  pdfkeywords={entropy function, Choquet simplex, ergodic entropy spectrum,
    Katok's intermediate-entropy conjecture, isolated invariant set}
}

\numberwithin{equation}{section}
\setlist[enumerate]{leftmargin=2.4em}
\newtheorem{theorem}{Theorem}[section]
\newtheorem{proposition}[theorem]{Proposition}
\newtheorem{lemma}[theorem]{Lemma}
\newtheorem{corollary}[theorem]{Corollary}
\newtheorem*{conjecture}{Conjecture}
\theoremstyle{remark}
\newtheorem{remark}[theorem]{Remark}

\newcommand{\id}{\operatorname{id}}
\newcommand{\htop}{h_{\mathrm{top}}}
\newcommand{\Me}{\mathcal M_{\mathrm e}}
\newcommand{\Minv}{\mathcal M}
\newcommand{\He}{\mathscr H_{\mathrm e}}
\newcommand{\dd}{\,\mathrm d}

\title[Entropy functions on closed manifolds]
{Realizing Prescribed Entropy Functions\\
by Smooth Diffeomorphisms\\
of Closed Manifolds\\
of Dimension at Least Three}

\author{Wanshan Lin}
\address{School of Mathematical Science, Fudan University, Shanghai 200433, People's Republic of China}
\email{wanshanlin@fudan.edu.cn}

\author{Xueting Tian}
\address{School of Mathematical Science, Fudan University, Shanghai 200433, People's Republic of China}
\email{xuetingtian@fudan.edu.cn}


\begin{document}

\begin{abstract}
Let $M$ be a closed smooth manifold of dimension $d\geq3$.  Given a compact metrizable
Choquet simplex $\mathscr S$ and a bounded nonnegative affine upper
semicontinuous function $\mathfrak e$ on $\mathscr S$, we construct a
$C^\infty$ diffeomorphism $h$ of $M$, isotopic to $\id_M$ and supported in
an embedded $d$-dimensional solid torus $D^{d-1}\times S^1$, with an
isolated minimal invariant Cantor set $K$.
The invariant-measure simplex of $h|_K$ is affinely homeomorphic to
$\mathscr S$ with entropy function $\mathfrak e$, whereas every ergodic
$h$-invariant measure not supported on $K$ is a Dirac measure at a fixed
point.  Consequently, the set of measure-theoretic entropies of ergodic
$h$-invariant probability measures and the topological entropy of $h$ are
\[
 \He(h)=\{0\}\cup\mathfrak e(\operatorname{ex}\mathscr S),
 \qquad
 \htop(h)=\max_{p\in\mathscr S}\mathfrak e(p).
\]
The map $h$ is $C^\infty$-approximable by zero-entropy diffeomorphisms
isotopic to $\id_M$.  Taking $\mathscr S$ to be a singleton yields
counterexamples to Katok's intermediate-entropy conjecture on every such $M$.
We also construct such counterexamples $h_j$ and numbers $c_j>0$ with
$h_j\to\id_M$ in $C^\infty$, $c_j\to0$, and
\[
 \He(h_j)=\{0,c_j\},
 \qquad
 \htop(h_j)=c_j.
\]
Hence the intermediate-entropy property is not $C^\infty$ open among
diffeomorphisms isotopic to the identity.
\end{abstract}

\maketitle

\section{Introduction}

Whether positive topological entropy forces a rich supply of invariant
measures is a longstanding question.  In general topological dynamics the
answer is negative: positive-entropy strictly ergodic examples include
subshifts \cite{Grillenberger1973,HahnKatznelson1967} and homeomorphisms of
compact manifolds \cite{BeguinCrovisierLeRoux2007}.  Herman asked whether
minimal or strictly ergodic diffeomorphisms can have positive topological
entropy~\cite[p.~141]{Katok1980}.  He subsequently answered the minimality
part affirmatively by constructing a real-analytic minimal diffeomorphism of
positive entropy on a compact four-manifold \cite{Herman1981}.  These results
show that positive topological entropy is compatible with minimality and even
with strict ergodicity.  Our focus is the finer question of which values of
measure-theoretic entropy are realized by ergodic measures.

Let $Z$ be a compact metric space and $T:Z\to Z$ continuous.  We write
$\Minv(T)$ for the $T$-invariant Borel probability measures, $\Me(T)$ for
the ergodic ones, and
\[
   \He(T):=\{h_\mu(T):\mu\in\Me(T)\}
\]
for its \emph{ergodic entropy spectrum}.  The variational principle only
asserts that $\sup\He(T)=\htop(T)$; by itself it gives no intermediate-value
property for the extreme points of $\Minv(T)$.

At the topological level, the entropy function is remarkably flexible.
Downarowicz and Serafin proved that every bounded nonnegative affine upper
semicontinuous function on a compact metrizable Choquet simplex is the entropy
function on the invariant-measure simplex of a finite-alphabet minimal
Toeplitz subshift
\cite[Theorem~1]{DownarowiczSerafin2003}.  Using this realization theorem,
Konieczny, Kupsa, and Kwietniak constructed, for every Polish space $P$, a
minimal Toeplitz subshift whose space of ergodic measures, with the weak$^*$
topology, is homeomorphic to $P$ and which has a unique positive-entropy
ergodic measure
\cite[Theorem~11]{KoniecznyKupsaKwietniak2018}.

Given any such simplex--entropy-function pair, our main theorem realizes it
on an isolated minimal Cantor subsystem of a $C^\infty$ diffeomorphism of
any closed manifold of dimension at least three.  The construction is local:
in dimension $d$ it is supported in an
embedded copy of $D^{d-1}\times S^1$, and the only additional
ergodic measures are Dirac measures at fixed points and hence have zero
entropy.

This realization yields counterexamples to Katok's intermediate-entropy
conjecture.  We recall its standard formulation \cite[p.~2]{Sun2025}:

\begin{conjecture}[Katok]
For every $C^2$ diffeomorphism $f$ of a compact Riemannian manifold,
\[
       [0,\htop(f))\subseteq \He(f).
       \tag{IE}\label{IE}
\]
\end{conjecture}

In dimension one, every diffeomorphism of a compact one-dimensional manifold
has zero topological entropy, so \eqref{IE} is vacuous.  For
$C^{1+\alpha}$ surface diffeomorphisms, where $\alpha>0$, Katok's
horseshoe approximation theory implies \eqref{IE}
\cite{Katok1980}; see also
\cite[Theorem~S.5.9]{KatokHasselblatt1995}.  Thus dimension three is the
first dimension in which a counterexample can occur.

More precisely, taking $\mathscr S=\{\ast\}$ and $\mathfrak e(\ast)=H>0$ in
\cref{thm:main} yields, on every closed smooth manifold of dimension $d\geq3$,
a diffeomorphism violating \eqref{IE}; hence this dimensional range is sharp.

Moreover, using primitive renormalization near the identity
\cite[Theorem~A]{BergerGourmelonHelfter2025}, we show in
\cref{thm:near-identity} that such diffeomorphisms accumulate at $\id_M$ in the
$C^\infty$ topology.  Thus the class satisfying \eqref{IE} is not
$C^\infty$ open.  As \cref{rem:not-locally-dense} explains, this accumulation
cannot be strengthened to density in any neighborhood of $\id_M$.

Under additional dynamical hypotheses, however, \eqref{IE} holds.
For smooth systems, Sun proved a relative result for
skew products preserving an ergodic measure of positive entropy: over an
invertible ergodic zero-entropy base, nonzero
Lyapunov exponents for the $C^{1+\alpha}$ fiber maps yield ergodic measures of
every entropy up to that of the given measure \cite{Sun2010}.  For partially
hyperbolic diffeomorphisms with one-dimensional center, positive results
include absolutely partially hyperbolic diffeomorphisms of $\mathbb T^3$
homotopic to a hyperbolic toral automorphism \cite{Ures2012}.
Another result gives a $C^1$-open and dense subfamily within the open class of
systems having
hyperbolic periodic points of different indices and robustly minimal strong
foliations \cite{YangZhang2020}.  Mostly expanding partially hyperbolic
diffeomorphisms with a one-dimensional center bundle and minimal strong
stable foliation provide another class \cite{Zhang2025}; finer joint
flexibility of entropy and the center Lyapunov exponent holds when stable and
unstable blender-horseshoes coexist and the strong stable and strong unstable
foliations are both minimal \cite{DiazGelfertRamsZhang2026}.  An
equilibrium-state approach also verifies the conjecture for a class of
Ma\~n\'e diffeomorphisms \cite{Sun2021}.

For flows, every star vector field
has the intermediate-entropy property \cite{LiShiWangWang2020}; Arbieto,
Oprocha, and Rego strengthened this to \emph{entropy flexibility} by realizing
each intermediate value on a strictly ergodic subsystem having that value as
its topological entropy
\cite{ArbietoOprochaRego2025}.

A complementary topological route proceeds through ergodic universality.
Quas and Soo \cite{QuasSoo2016} and Burguet \cite{Burguet2020} obtained
universality results for systems with specification-type properties, while
Chandgotia and Meyerovitch developed a broader criterion based on flexible
marker sequences \cite{ChandgotiaMeyerovitch2021}.  Since these universality
results realize every invertible non-atomic ergodic
probability-preserving system of entropy below the topological entropy, they
imply (IE).  Applications of this universality approach give the full-interval
conclusion for left translations by Haar-almost every $g\in G$ on $G/\Gamma$,
where $G$ is a connected semisimple Lie group and $\Gamma$ is a cocompact
lattice \cite{GuanSunWu2017}, and for every affine transformation of a
nilmanifold having a periodic point \cite{HuangXuXu2021}.

A third route proceeds through weak$^{*}$ approximation of invariant measures
by ergodic measures with prescribed entropy.  If $g$
is either a transitive system with shadowing and an upper
semicontinuous entropy map \cite[Corollary~C(2)]{LiOprocha2018}, or an
asymptotically entropy expansive system with the approximate product property
\cite[Theorem~1.4(2)]{Sun2025}, then, for every
$0\leq c<\htop(g)$, the ergodic measures of entropy $c$ are residual, with
respect to the relative weak$^{*}$ topology, among the invariant measures of
entropy at least $c$.  The same residual conclusion holds for shift spaces
whose languages are edit
approachable by a collection with $(W)$-specification
\cite[Theorem~3.1]{JiChenLin2022}.
Liu, Chen, and Zhou extended Sun's approximate-product and
entropy-genericity results to continuous actions of $\mathbb Z^d$ and
$\mathbb Z_+^d$ on compact metric spaces
\cite[Theorem~1.4]{LiuChenZhou2026}.
In subsequent work, they proved that asymptotically entropy
expansive actions of countably infinite discrete amenable groups with the
specification property are entropy-generic and admit ergodic measures of every
entropy in the closed interval from zero to the topological entropy
\cite[Theorems~1.1 and~1.4]{LiuChenZhouAmenable2026}.

We now state the realization theorem announced above and its consequence for
Katok's intermediate-entropy conjecture.  Here and throughout, a
\emph{closed manifold} means a nonempty compact smooth manifold without
boundary, not necessarily connected.
For $d\geq2$, a
\emph{$d$-dimensional solid torus} means $D^{d-1}\times S^1$ up to
diffeomorphism, where $D^{d-1}$ is the closed $(d-1)$-disk and
$S^1=\mathbb R/(2\pi\mathbb Z)$.

A compact invariant set $C$ of a homeomorphism $T:Z\to Z$ is
\emph{isolated} if there is an open neighborhood $U$ of $C$ such that
\[
       C=\bigcap_{n\in\mathbb Z}T^n(U).
\]

For a self-map $g:X\to X$, write
\[
 \operatorname{Fix}(g)=\{x\in X:g(x)=x\}.
\]

For a diffeomorphism $g$ of a smooth manifold $X$, write
\[
 \operatorname{supp}(g)
 =\overline{\{x\in X:g(x)\ne x\}},
\]
where the closure is taken in $X$.

If $\psi:X\to Y$ is a Borel map and $\mu$ is a Borel probability measure on
$X$, we write $\psi_*\mu$ for the pushforward measure, defined by
\[
 (\psi_*\mu)(B)=\mu\bigl(\psi^{-1}(B)\bigr)
\]
for every Borel set $B\subset Y$.  For a family $\mathcal A$ of probability
measures, we use the convention
\[
 \psi_*\mathcal A=\{\psi_*\mu:\mu\in\mathcal A\}.
\]

A convex subset $\mathcal F$ of a convex set $\mathcal C$ is called a
\emph{face} if, whenever $x,y\in\mathcal C$ and $0<t<1$ satisfy
\[
 tx+(1-t)y\in\mathcal F,
\]
one has $x,y\in\mathcal F$.  It is a \emph{closed face} if it is also closed
in $\mathcal C$.  For convex sets of probability measures, closedness is
always understood with respect to the weak$^*$ topology.

For every closed smooth manifold $M$, equip
$\operatorname{Diff}^\infty(M)$ with the $C^\infty$ topology.  We write
$\operatorname{Diff}_0^\infty(M)$ for the subgroup of diffeomorphisms smoothly
isotopic to $\id_M$; explicitly, $f\in\operatorname{Diff}_0^\infty(M)$ if
there exists a family $(f_t)_{t\in[0,1]}\subset\operatorname{Diff}^\infty(M)$
such that $(t,x)\mapsto f_t(x)$ is smooth, $f_0=\id_M$, and $f_1=f$.

\subsection{Main results}

\begin{theorem}
\label{thm:main}
Let $M$ be a closed smooth manifold of dimension $d\geq3$.  Let
$\mathscr S$ be a nonempty
compact metrizable Choquet simplex, and write $\operatorname{ex}\mathscr S$
for its set of extreme points.  Let
$\mathfrak e:\mathscr S\to[0,\infty)$ be bounded, affine, and upper
semicontinuous.  There are an embedded $d$-dimensional solid torus
$N\subset M$, a
diffeomorphism $h\in\operatorname{Diff}_0^\infty(M)$, and a minimal invariant
Cantor set $K\subset\operatorname{int}N$ with the following properties.

\begin{enumerate}[label=\textup{(\roman*)}]
\item The map $h$ is the identity on $M\setminus\operatorname{int}N$, the set
      $K$ is isolated, and $K\cap\operatorname{Fix}(h)=\varnothing$.  Every
      $h$-invariant probability is supported on
      $\operatorname{Fix}(h)\cup K$.  The set
      \[
         \mathcal F_K:=\{\mu\in\Minv(h):\mu(K)=1\}
      \]
      is a closed face of $\Minv(h)$, and there is an affine homeomorphism
      $\Phi:\mathscr S\to\mathcal F_K$ such that
      \[
          h_{\Phi(p)}(h)=\mathfrak e(p)
          \qquad(p\in\mathscr S).
      \]
\item The ergodic measures, ergodic entropy spectrum, and topological entropy
      are
      \[
      \Me(h)=\{\delta_z:z\in\operatorname{Fix}(h)\}
              \cup\Phi(\operatorname{ex}\mathscr S),
      \]
      \[
          \He(h)=\{0\}\cup\mathfrak e(\operatorname{ex}\mathscr S),
          \qquad
          \htop(h)=\max_{p\in\mathscr S}\mathfrak e(p).
      \]
\item There is a sequence $(h_n)_{n\geq1}$ in
      $\operatorname{Diff}_0^\infty(M)$, each equal to the identity on
      $M\setminus\operatorname{int}N$, such
      that
      \[
          h_n\longrightarrow h
          \quad\text{in the $C^\infty$ topology},
          \qquad
          \htop(h_n)=0
          \quad(n\geq1).
      \]
\end{enumerate}
\end{theorem}

The following corollary is an immediate consequence of \cref{thm:main}.

\begin{corollary}
\label[corollary]{cor:compact-spectra}
Let $M$ be any closed smooth manifold of dimension at least three.
For every nonempty compact set
$A\subset[0,\infty)$, there exists
$h\in\operatorname{Diff}_0^\infty(M)$ such that
\[
       \He(h)=\{0\}\cup A,
       \qquad
       \htop(h)=\max A.
\]
\end{corollary}

The diffeomorphism in \cref{cor:compact-spectra} may be chosen with the
localized support, isolated minimal Cantor set, and zero-entropy approximation
in \cref{thm:main}.  For $A=\{H\}$, where $H>0$, take
$\mathscr S=\{\ast\}$ and
$\mathfrak e(\ast)=H$.  The face $\mathcal F_K$ then consists of one measure
$\nu_H$.  Thus
\[
       \Me(h)=\{\delta_z:z\in\operatorname{Fix}(h)\}\cup\{\nu_H\},
       \qquad
       \He(h)=\{0,H\},
       \qquad
       \htop(h)=H,
\]
where $\nu_H(K)=1$ and $h_{\nu_H}(h)=H$.  Hence no $a\in(0,H)$ is the
entropy of an ergodic measure.

Thus Katok's intermediate-entropy conjecture fails on every closed smooth
manifold of dimension at least three.

Moreover, counterexamples whose ergodic entropy spectrum consists of two
points accumulate at $\id_M$ in the $C^\infty$ topology, with their positive
entropy values tending to zero.

\begin{theorem}
\label{thm:near-identity}
Let $M$ be a closed smooth manifold of dimension $d\geq3$.  There
are an embedded $d$-dimensional solid torus $N\subset M$, diffeomorphisms
$h_j\in\operatorname{Diff}_0^\infty(M)$, and numbers $c_j>0$ such that
$h_j\to\id_M$ in the $C^\infty$ topology and $c_j\to0$.
For every $j$, the map $h_j$ is the identity on
$M\setminus\operatorname{int}N$, and there is an isolated minimal invariant
Cantor set $K_j\subset\operatorname{int}N$ carrying a unique invariant
probability $\nu_j$.  Moreover,
\[
 \begin{aligned}
 \Me(h_j)&=\{\delta_z:z\in\operatorname{Fix}(h_j)\}\cup\{\nu_j\},
 & h_{\nu_j}(h_j)&=c_j,\\
 \He(h_j)&=\{0,c_j\},
 & \htop(h_j)&=c_j.
 \end{aligned}
\]
Consequently, the set of $C^\infty$ diffeomorphisms satisfying
\eqref{IE} is not open in the $C^\infty$ topology.
\end{theorem}

\begin{remark}
\label[remark]{rem:not-locally-dense}
The conclusion of \cref{thm:near-identity} cannot be strengthened to density
of the counterexamples in a $C^\infty$ neighborhood of $\id_M$.  Indeed,
every such $M$ admits a Morse--Smale gradient-like vector field $X$; see
Smale~\cite[Theorems~A and~B]{Smale1961Gradient}.  Let
$(\varphi_X^t)_{t\in\mathbb R}$ be its flow.  Since $X$ is Morse--Smale, it
has finitely many equilibria, all hyperbolic.  Here gradient-like means that
$X$ admits a smooth function $V\colon M\to\mathbb R$ satisfying
$dV_x(X(x))>0$ whenever $X(x)\ne0$.  Hence
$V(\varphi_X^t(x))>V(x)$ for every nonsingular $x$ and every $t>0$, so every
nonsingular point is wandering for $\varphi_X^t$.  Thus its nonwandering set
is precisely the set of equilibria of $X$.  For every $t>0$, these equilibria
are hyperbolic fixed points of the time-$t$ map, and their stable and unstable
manifolds are those of the flow.  Hence, for every pair of equilibria $p,q$,
the unstable manifold of $p$ and the stable manifold of $q$ intersect
transversely.  Consequently, $\varphi_X^t$ is a Morse--Smale diffeomorphism
for every $t>0$.  Smooth dependence of the flow on time gives
\[
 \varphi_X^t\longrightarrow\id_M
 \quad\text{in the $C^\infty$ topology as }t\downarrow0.
\]

Given a $C^\infty$ neighborhood $\mathcal U$ of $\id_M$, choose $t>0$ with
$\varphi_X^t\in\mathcal U$.  Morse--Smale diffeomorphisms are $C^1$
structurally stable \cite{PalisSmale1970}; see also
\cite[p.~273]{FranksShub1981}.  Let $\mathcal W$ be a $C^1$
structural-stability neighborhood of $\varphi_X^t$.  Then
\[
 \mathcal V=\mathcal U\cap\mathcal W\cap\operatorname{Diff}^\infty(M)
\]
is a nonempty $C^\infty$ open set whose elements are topologically conjugate
to $\varphi_X^t$.  Since a Morse--Smale diffeomorphism has finite nonwandering
set, every element of $\mathcal V$ has topological entropy zero and therefore
satisfies \eqref{IE} vacuously.  Thus
the counterexamples are not dense in any neighborhood of $\id_M$.
\end{remark}

\medskip
\noindent\textbf{Proof strategy and organization.}
For \cref{thm:main}, the Downarowicz--Serafin realization theorem supplies an
infinite minimal Toeplitz subshift whose invariant-measure simplex and entropy
function realize the prescribed pair $(\mathscr S,\mathfrak e)$.  We embed this
subshift as an invariant Cantor subsystem of an orientation-preserving $C^\infty$ surface
diffeomorphism whose support and horseshoe both lie in the interior of a
two-dimensional disk; hence the resulting disk map is the identity near the
boundary.  Smale's contractibility theorem for relative-boundary square
diffeomorphisms, together with its smooth-parametric form
\cite[Theorems~B and~4]{Smale1959}, provides an isotopy on a square inside the
disk from the identity to the disk map, relative to the boundary; we thicken
this isotopy in $d-3$ auxiliary directions.

Coupling the resulting $(d-1)$-dimensional disk dynamics to a
controlled drift in an additional circle coordinate produces a local map $H$
on $D^{d-1}\times S^1$ and the Lyapunov function $L=-\cos\theta$ with
\[
 L\circ H\geq L,
 \qquad
 \{L\circ H=L\}=K\sqcup\operatorname{Fix}(H).
\]
The set $K$ is compact and $H$-invariant, and $H|_K$ is conjugate to the
original Toeplitz system.  Thus the invariant-measure simplex of $H|_K$ and
its entropy function realize $(\mathscr S,\mathfrak e)$.  Integrating the
nonnegative coboundary $L\circ H-L$ against an $H$-invariant probability
forces its support to lie in this equality set.  Since $L|_K\equiv0$ and $K$
is separated from the fixed-point set, the local equality-set criterion shows
that $K$ is isolated.  Moreover, \cref{prop:localized-isolation}, together
with the preceding conjugacy and \cref{lem:entropy-source}, gives the local
ergodic-measure classification and the exact formulas
$\He(H)=\{0\}\cup\mathfrak e(\operatorname{ex}\mathscr S)$ and
$\htop(H)=\max_{p\in\mathscr S}\mathfrak e(p)$.  Drift perturbations with
parameters tending to zero
give $H_n\to H$ in $C^\infty$ with
$\{L\circ H_n=L\}=\operatorname{Fix}(H_n)$;
the same support argument and the variational principle then yield
$\htop(H_n)=0$.

A fixed embedding of the solid torus into $M$, together with
the common boundary collar, allows the local maps and isotopies to be extended
by the identity.  Denote the resulting global maps by $h$ and $h_n$.  Then
$h_n\to h$ in $C^\infty$, and all these maps are smoothly isotopic to $\id_M$.
The image of $K$, still denoted by $K$, remains isolated, and
$\mathcal F_K\subset\Minv(h)$ is a closed face affinely homeomorphic to
$\mathscr S$.  The invariance of the solid torus transfers the local
ergodic-measure classification and entropy formulas from $H$ to $h$, with the
fixed exterior adding only zero-entropy Dirac measures to $\Me(h)$.  For each
$h_n$, every ergodic invariant probability is either transported from an
ergodic $H_n$-invariant probability or is a Dirac mass in the fixed exterior.
Since $\htop(H_n)=0$, the ergodic variational principle gives $\htop(h_n)=0$.

For \cref{cor:compact-spectra}, take $\mathscr S=\mathcal P(A)$, where
$\mathcal P(A)$ is the simplex of Borel probability measures on $A$, and set
$\mathfrak e(\lambda)=\int_A a\,\dd\lambda(a)$; then
$\operatorname{ex}\mathcal P(A)=\{\delta_a:a\in A\}$ gives
$\mathfrak e(\operatorname{ex}\mathscr S)=A$.

For \cref{thm:near-identity}, fix $H_*>0$.  Applying
\cref{lem:entropy-source} to the one-point simplex with entropy function equal
to $H_*$ gives an infinite
minimal uniquely ergodic Toeplitz subshift whose unique invariant measure has
entropy $H_*$.  We embed this subshift as a $G$-invariant Cantor set $Z$ inside
a horseshoe of a fixed orientation-preserving disk diffeomorphism $G:D\to D$;
the ambient horseshoe also contains a hyperbolic fixed point.  Applying the
primitive renormalization theorem of Berger, Gourmelon, and Helfter
\cite[Definition~1.1 and Theorem~A]{BergerGourmelonHelfter2025} along a
shrinking $C^\infty$ neighborhood basis gives $q_j\to\id_D$ in $C^\infty$,
return times $N_j\geq2$, and smooth embeddings $\psi_j:D\to D$ such that
\[
 q_j^{N_j}\circ\psi_j=\psi_j\circ G,
 \qquad
 q_j^i\bigl(\psi_j(D)\bigr)\cap\psi_j(D)=\varnothing
 \quad(0<i<N_j).
\]
Characteristic-polynomial invariance under derivative conjugacy at this
hyperbolic fixed point, together with $q_j\to\id_D$, forces $N_j\to\infty$.

\Cref{lem:relative-extension} then supplies diffeomorphisms $\widehat q_j$ on
a single fixed larger disk that preserve the displayed renormalization data,
equal the identity on a fixed boundary collar, and still converge to the
identity in $C^\infty$.  The preserved return data
give an $N_j$-level cycle of $\psi_j(Z)$, from which
\cref{lem:cyclic-renormalization} builds a minimal uniquely ergodic tower.  Its
invariant measure assigns
mass $1/N_j$ to each level, and the return after $N_j$ iterates on every level
is conjugate to $G|_Z$; hence the entropy formula for iterates gives
$H_*/N_j\to0$.

\Cref{lem:small-localized-implantation} turns these towers into solid-torus
diffeomorphisms converging to the identity.  A fixed embedding into $M$ and
extension by the identity then gives $h_j\to\id_M$ in $C^\infty$.  The
pointwise fixed complement supplies the entropy value $0$, while the unique
tower measure supplies $H_*/N_j$.  The Lyapunov equality-set support argument
shows that every other ergodic measure is a fixed-point Dirac mass.  Hence
$\He(h_j)=\{0,H_*/N_j\}$, and the variational principle gives
$\htop(h_j)=H_*/N_j\to0$.

\Cref{sec:inputs} supplies the symbolic, smooth, and measure-theoretic inputs.
\Cref{sec:localization} develops the solid-torus construction and proves
\cref{thm:main,cor:compact-spectra}, while \cref{sec:near-identity} proves
\cref{thm:near-identity}.

\section{Preliminaries}
\label{sec:inputs}

We record four inputs: a Toeplitz realization of the prescribed entropy data,
a disk-supported smooth horseshoe carrying the resulting subshift, a smooth
function with a prescribed zero set, and the support principle for the
nonnegative-coboundary argument in \cref{sec:localization}.

\subsection{Toeplitz realizations of entropy functions}

For a finite alphabet $\mathcal A$, let
$\sigma:\mathcal A^{\mathbb Z}\to\mathcal A^{\mathbb Z}$ denote the left
shift, $(\sigma y)_n=y_{n+1}$.  A \emph{two-sided subshift} is a closed set
$Y\subset\mathcal A^{\mathbb Z}$ satisfying $\sigma(Y)=Y$.

\begin{lemma}
\label[lemma]{lem:entropy-source}
Let $\mathscr S$ be a nonempty compact metrizable Choquet simplex, and let
$\mathfrak e:\mathscr S\to[0,\infty)$ be bounded, affine, and upper
semicontinuous.  There are an integer $k\geq2$, a minimal two-sided Toeplitz
subshift $Y\subset\{1,\dots,k\}^{\mathbb Z}$ with no isolated points, and an
affine homeomorphism
\[
       \phi:\mathscr S\longrightarrow\Minv(\sigma|_Y)
\]
such that
\[
       h_{\phi(p)}(\sigma|_Y)=\mathfrak e(p)
       \qquad(p\in\mathscr S).
\]
Moreover,
\[
       \phi(\operatorname{ex}\mathscr S)=\Me(\sigma|_Y),
       \qquad
       \htop(\sigma|_Y)=\max_{p\in\mathscr S}\mathfrak e(p).
\]
\end{lemma}

\begin{proof}
Downarowicz and Serafin \cite[Theorem~1]{DownarowiczSerafin2003} provide a
finite-alphabet Toeplitz subshift $(Y,\sigma)$ and the asserted affine
homeomorphism $\phi$ preserving the entropy function.  This subshift is a
minimal symbolic almost one-to-one extension of a fixed nonperiodic odometer
\cite[pp.~222--223]{DownarowiczSerafin2003}.  Hence $Y$ is infinite.  An
infinite minimal compact system has no isolated point, and a nonperiodic
subshift uses at least two symbols; relabel its alphabet as
$\{1,\ldots,k\}$.

An affine homeomorphism preserves extreme points, while the extreme points of
an invariant-measure simplex are precisely the ergodic measures.  Hence
$\phi(\operatorname{ex}\mathscr S)=\Me(\sigma|_Y)$.  Finally, upper
semicontinuity and compactness give a
maximum for $\mathfrak e$, and the variational principle together with the
surjectivity of $\phi$ gives
\[
 \htop(\sigma|_Y)
 =\sup_{\nu\in\Minv(\sigma|_Y)}h_\nu(\sigma|_Y)
 =\max_{p\in\mathscr S}\mathfrak e(p).\qedhere
\]
\end{proof}

\subsection{Realizing subshifts on the sphere}

For $r\geq2$, write $\Sigma_r=\{1,\ldots,r\}^{\mathbb Z}$ for the full
two-sided $r$-shift.  For $k\geq2$, we use the shorthand
\emph{$k$-leg horseshoe} for a locally maximal hyperbolic set of a surface
diffeomorphism whose stable and unstable bundles are one-dimensional and
whose restricted dynamics is conjugate to $(\Sigma_k,\sigma)$.

\begin{lemma}
\label[lemma]{lem:horseshoe-embedding}
Let $\varnothing\ne Y\subset\Sigma_k$ be a two-sided subshift, where $k\geq2$.
There exist an orientation-preserving $C^\infty$ diffeomorphism
$f:S^2\to S^2$, a compact $f$-invariant
topologically zero-dimensional set $X$,
and a homeomorphism $\chi:Y\to X$ satisfying
\[
       \chi\circ\sigma=f\circ\chi.
\]
Moreover, $X$ lies in a $k$-leg horseshoe $\Lambda$ of $f$, and $f$ may be
chosen so that $\operatorname{supp}(f)\cup\Lambda$ is contained in the
interior of one closed disk in $S^2$.  If $Y$ has no isolated points, then
$X$ is a Cantor set.
\end{lemma}

\begin{proof}
We first construct and code a disk-supported two-leg horseshoe.  We then pass
to an iterate with at least $k$ block symbols and verify the local maximality
of the selected $k$-symbol subhorseshoe.

Let $D\subset S^2$ be a closed disk.  We use the standard
orientation-preserving affine two-strip horseshoe model described in
\cite[\S2.5c]{KatokHasselblatt1995}.  As noted in
\cite[Appendix~A.1, especially p.~1786]{Buzzi2014}, this model admits a
$C^\infty$ realization on $D$ that equals the identity near $\partial D$.
Accordingly, choose such a diffeomorphism $G:D\to D$ and a closed rectangle
$R\subset\operatorname{int}D$ in the two-strip horseshoe configuration.
Since $G=\id$ near $\partial D$, it extends by the
identity on $S^2\setminus D$ to a $C^\infty$ diffeomorphism of $S^2$, still
denoted by $G$.  Write $R=I^u\times J^s$ in product coordinates.

There are
pairwise disjoint compact subintervals $I_1^u,I_2^u$ of
$\operatorname{int}I^u$ and pairwise disjoint compact subintervals
$J_1^s,J_2^s$ of $\operatorname{int}J^s$ such that, for
$V_i=I_i^u\times J^s$,
\[
 \begin{aligned}
 R\cap G^{-1}(R)&=V_1\cup V_2,\\
 G(V_i)&=I^u\times J_i^s \qquad (i=1,2).
 \end{aligned}
\]
Thus the $V_i$ are vertical strips crossing $R$, while their images are
horizontal strips crossing $R$.

Choose the model so that, for $i=1,2$,
\[
 G|_{V_i}\colon I_i^u\times J^s \longrightarrow I^u\times J_i^s
\]
is an affine diffeomorphism with a common expansion factor $a>1$ in the
$I^u$-coordinate and a common contraction factor $0<\tau<1$ in the
$J^s$-coordinate.  Since $G$ is orientation preserving, there are signs
$\varepsilon_i,\delta_i\in\{-1,1\}$ with $\varepsilon_i\delta_i=1$ such that
\[
 DG_z(\xi,\eta)=(\varepsilon_i a\xi,\delta_i\tau\eta)
 \qquad (z\in V_i).
\]

Choose a Riemannian metric on $S^2$ that agrees with the Euclidean product
metric on a neighborhood of $R$.  This merely simplifies the cone estimates:
all Riemannian metrics on the compact manifold $S^2$ are uniformly equivalent,
so uniform hyperbolicity is unaffected by this choice, although the constants
in the estimates may change.  Choose $\alpha>0$ so small that
\[
 \frac{\min\{a,\tau^{-1}\}}{\sqrt{1+\alpha^2}}>1.
\]
Define constant horizontal and vertical cone fields on the
product-coordinate neighborhood of $R$ by
\[
 \begin{aligned}
 C_z^u&=\{(\xi,\eta):|\eta|\leq\alpha|\xi|\},\\
 C_z^s&=\{(\xi,\eta):|\xi|\leq\alpha|\eta|\}.
 \end{aligned}
\]

If $z\in V_i$ and $0\ne v=(\xi,\eta)\in C_z^u$, then
\[
 \frac{|(DG_zv)_2|}{|(DG_zv)_1|}
 \leq\frac{\tau}{a}\alpha<\alpha,
 \qquad
 \lVert DG_zv\rVert
 \geq\frac{a}{\sqrt{1+\alpha^2}}\lVert v\rVert.
\]
Likewise, if $z\in V_i$, $y=G(z)$, and
$0\ne w=(\xi,\eta)\in C_y^s$, then
\[
 D(G^{-1})_y w
   =(\varepsilon_i a^{-1}\xi,\delta_i\tau^{-1}\eta),
\]
and therefore
\[
 \frac{|(D(G^{-1})_yw)_1|}{|(D(G^{-1})_yw)_2|}
 \leq\frac{\tau}{a}\alpha<\alpha,
 \qquad
 \lVert D(G^{-1})_yw\rVert
 \geq\frac{\tau^{-1}}{\sqrt{1+\alpha^2}}\lVert w\rVert.
\]
Fix
\[
 1<\lambda<\frac{\min\{a,\tau^{-1}\}}{\sqrt{1+\alpha^2}}.
\]

\noindent The preceding estimates give
\[
\begin{aligned}
 DG_x(C_x^u\setminus\{0\})&\subset\operatorname{int}C_{G(x)}^u,
 &\lVert DG_xv\rVert&\geq\lambda\lVert v\rVert,\\
 D(G^{-1})_y(C_y^s\setminus\{0\})&\subset
       \operatorname{int}C_{G^{-1}(y)}^s,
 &\lVert D(G^{-1})_yw\rVert&\geq\lambda\lVert w\rVert.
\end{aligned}
\]
The first line holds for $x\in V_1\cup V_2$ and
$0\ne v\in C_x^u$, and the second for $y\in G(V_1)\cup G(V_2)$ and
$0\ne w\in C_y^s$.

Put
\[
 \Lambda=\bigcap_{n\in\mathbb Z}G^n(R).
\]
The set $\Lambda$ is compact and $G$-invariant.  If $x\in\Lambda$, then
$x,G(x)\in R$, and hence
\[
       x\in R\cap G^{-1}(R)=V_1\cup V_2.
\]
Applying this observation to $G^{-1}(x)\in\Lambda$ gives
$x\in G(V_1)\cup G(V_2)$, so both cone estimates hold throughout $\Lambda$.

The cone criterion
\cite[Corollary~6.4.8]{KatokHasselblatt1995} therefore shows that
$\Lambda$ is hyperbolic and carries a continuous invariant splitting
\[
       T_\Lambda S^2=E^s\oplus E^u,
       \qquad \dim E^s=\dim E^u=1.
\]

Moreover, every $x\in\Lambda$ belongs both to $V_1\cup V_2$ and to
$G(V_1)\cup G(V_2)$.  Consequently,
\[
 \Lambda\subset
 (I_1^u\cup I_2^u)\times(J_1^s\cup J_2^s),
\]
and the set on the right is a compact subset of $\operatorname{int}R$.
Invariance of $\Lambda$ now gives
\[
 \Lambda=\bigcap_{n\in\mathbb Z}G^n(\operatorname{int}R),
\]
so it is locally maximal.

We next make the symbolic coding explicit.  For $\omega\in\Sigma_2$ and
$N\geq0$, let
\[
 K_N(\omega)=
 \{x\in R:G^n(x)\in V_{\omega_n}\text{ for }-N\leq n\leq N\}.
\]
For $i,j\in\{1,2\}$, the identities
$V_i=I_i^u\times J^s$ and $G(V_i)=I^u\times J_i^s$ give
\[
 G(V_i)\cap V_j=I_j^u\times J_i^s\neq\varnothing.
\]
Thus every horizontal strip $G(V_i)$ crosses every vertical strip $V_j$.
An induction on the length of the prescribed finite itinerary shows that
each $K_N(\omega)$ is a nonempty compact rectangle.

The definition also gives
$K_{N+1}(\omega)\subset K_N(\omega)$.  If $x,x'\in K_N(\omega)$ and
$x=(x^u,x^s)$, $x'=(x^{\prime u},x^{\prime s})$ in the product coordinates,
then the affine formulas on the common itinerary give
\[
 \begin{aligned}
 |x^u-x^{\prime u}|&\leq \operatorname{diam}(I^u)a^{-N},\\
 |x^s-x^{\prime s}|&\leq \operatorname{diam}(J^s)\tau^N.
 \end{aligned}
\]
Thus $\operatorname{diam}K_N(\omega)\to0$, uniformly in $\omega$.  The nested
sets consequently have a unique common point, denoted by $c(\omega)$.
Fix $n\in\mathbb Z$ and choose $N\geq|n|$.  Since
$c(\omega)\in K_N(\omega)$ and $n\in[-N,N]$, the definition of
$K_N(\omega)$ gives $G^n(c(\omega))\in V_{\omega_n}\subset R$.
Thus the full orbit of $c(\omega)$ remains in $R$, and consequently
\[
 c(\omega)\in\bigcap_{n\in\mathbb Z}G^{-n}(R)=\Lambda.
\]

Conversely, every $x\in\Lambda$ has a unique itinerary because $V_1$ and $V_2$
are disjoint.  Hence
$c:\Sigma_2\to\Lambda$ is bijective.  The uniform diameter estimate gives
continuity of $c$, and compactness of $\Sigma_2$ then makes $c$ a
homeomorphism.
For every $n\in\mathbb Z$,
\[
 G^n(G(c(\omega)))=G^{n+1}(c(\omega))
 \in V_{\omega_{n+1}}=V_{(\sigma\omega)_n}.
\]
Thus $G(c(\omega))$ has itinerary $\sigma\omega$.  By uniqueness of the point
with a prescribed itinerary,
\[
       c\circ\sigma=G\circ c.
\]

To obtain at least $k$ symbolic legs from the two-leg model, choose $m\geq1$
with $2^m\geq k$.  Set
\[
       \mathcal B=\{1,2\}^m,\qquad f=G^m,
\]
and define the block map by
\[
 B:\Sigma_2\longrightarrow\mathcal B^{\mathbb Z},
 \qquad
 B(\omega)_j
 =(\omega_{mj},\omega_{mj+1},\ldots,\omega_{mj+m-1}).
\]
If $\beta=(\beta_j)_{j\in\mathbb Z}\in\mathcal B^{\mathbb Z}$ and
$\beta_j=(\beta_{j,0},\ldots,\beta_{j,m-1})$, then
\[
       (B^{-1}\beta)_{mj+r}=\beta_{j,r}
       \qquad(j\in\mathbb Z,\ 0\leq r<m).
\]
Thus $B$ is a homeomorphism and $B\circ\sigma^m=\sigma\circ B$.

Hence
\[
       \widehat c=c\circ B^{-1}:\mathcal B^{\mathbb Z}\longrightarrow\Lambda
\]
conjugates the full shift on the alphabet $\mathcal B$ to $f|_\Lambda$; in
particular,
\[
       \widehat c\circ\sigma=f\circ\widehat c.
\]

Select $\mathcal A\subset\mathcal B$ with $|\mathcal A|=k$ and set
\[
       \Lambda_k=\widehat c(\mathcal A^{\mathbb Z}).
\]
The set $\Lambda_k$ is hyperbolic for $f$: the splitting over $\Lambda$ for
$G$ is also hyperbolic for $G^m$, and its restriction to the compact invariant
subset $\Lambda_k$ remains hyperbolic.  Its stable and unstable bundles are
the restrictions of $E^s$ and $E^u$, so both are one-dimensional.

To verify local maximality, let
$U=\operatorname{int}R$ and
\[
      U_m=\bigcap_{r=0}^{m-1}G^{-r}(U).
\]
Then $U_m$ is an open neighborhood of $\Lambda$.
Since $f$ is invertible, the intersection below is precisely the set of points
whose full $f$-orbit remains in $U_m$.  We have
\begin{equation}
\label{eq:power-isolation}
       \bigcap_{j\in\mathbb Z}f^j(U_m)=\Lambda.
\end{equation}
Indeed, if the full $f$-orbit of $x$ remains in $U_m$, then
$G^{mj+r}(x)\in U$ for every $j\in\mathbb Z$ and $0\leq r<m$.  These
integers $mj+r$ exhaust $\mathbb Z$, so the full $G$-orbit of $x$ remains in
$U$ and $x\in\Lambda$.  Conversely, if $x\in\Lambda$, then
$G^n(x)\in U$ for every $n\in\mathbb Z$.  Hence, for every $j\in\mathbb Z$
and $0\leq r<m$,
\[
       G^r(f^j(x))=G^{mj+r}(x)\in U.
\]
Thus $f^j(x)\in U_m$ for every $j\in\mathbb Z$, proving the reverse
inclusion.

The set
\[
 C=\widehat c\bigl(
      \{\omega\in\mathcal B^{\mathbb Z}:\omega_0\in\mathcal A\}
    \bigr)
\]
is clopen in $\Lambda$.  Since $\Lambda\setminus C$ is closed in $S^2$, the set
\[
       W=U_m\setminus(\Lambda\setminus C)
\]
is open.  Moreover, $C\subset W\subset U_m$, $W\cap\Lambda=C$, and
$\Lambda_k\subset C$.  Since $W\subset U_m$, every full $f$-orbit contained
in $W$ lies in $\Lambda$ by \eqref{eq:power-isolation}.

Because $W\cap\Lambda=C$, every iterate of such an orbit lies in $C$.
Moreover, if
$x=\widehat c(\omega)\in\Lambda$, then the conjugacy gives
$f^j(x)=\widehat c(\sigma^j\omega)$, and hence
$f^j(x)\in C$ if and only if $(\sigma^j\omega)_0=\omega_j\in\mathcal A$.
Consequently,
\[
 \bigcap_{j\in\mathbb Z}f^j(W)
 =\{x\in\Lambda:f^j(x)\in C\text{ for every }j\in\mathbb Z\}
 =\widehat c(\mathcal A^{\mathbb Z})=\Lambda_k.
\]
Thus $\Lambda_k$ is locally maximal.

After choosing a bijection
$\iota:\{1,\ldots,k\}\to\mathcal A$, its coordinatewise extension gives a
conjugacy
\[
       \psi=\widehat c\circ\iota^{\mathbb Z}:\Sigma_k\longrightarrow\Lambda_k.
\]
Together with the one-dimensional stable and unstable bundles identified
above, this proves that $\Lambda_k$ is a $k$-leg horseshoe for $f$.
Thus $\Lambda_k$ is the horseshoe required in the statement.

Put $X=\psi(Y)$ and $\chi=\psi|_Y$.  Then $X$ is
compact, $f(X)=X$, and $\chi\circ\sigma=f\circ\chi$.  Intersections of
cylinder sets with $Y$ form a clopen basis for $Y$, so $Y$ and $X$ are
topologically zero-dimensional.  If $Y$ has no isolated points, then $X$ is a
nonempty compact metrizable topologically zero-dimensional perfect space, and
therefore is a Cantor set.
The conjugacy also gives
\[
       \htop(f|_X)=\htop(\sigma|_Y).
\]
Finally, $f=G^m$ is orientation preserving.  Since $G$ is the identity near
$\partial D$ and on $S^2\setminus D$, and since
$\Lambda_k\subset\Lambda\subset\operatorname{int}D$, we also have
\[
 \operatorname{supp}(f)\cup\Lambda_k\subset\operatorname{int}D.
\]

\end{proof}

\subsection{Smooth functions with prescribed zero sets}

\begin{lemma}
\label[lemma]{lem:zero-set}
\cite[Theorem~2.29]{Lee2013}
If $A$ is a closed subset of a closed smooth manifold $Q$, then there exists
$\rho\in C^\infty(Q,[0,\infty))$ such that $\rho^{-1}(0)=A$.
\end{lemma}

\subsection{Nonnegative coboundaries}

\begin{lemma}
\label[lemma]{lem:coboundary}
Let $T:Z\to Z$ be a continuous map of a compact metric space, let
$L:Z\to\mathbb R$ be continuous, and suppose that $L\circ T-L\geq0$.
Then every $T$-invariant Borel probability $\mu$ satisfies
\[
 \operatorname{supp}\mu
 \subseteq\{z\in Z:L(Tz)=L(z)\}.
\]
\end{lemma}

\begin{proof}
For $\mu\in\Minv(T)$, invariance gives
\[
 \int_Z(L\circ T-L)\dd\mu=0.
\]
The integrand is continuous and nonnegative, so its closed zero set has full
$\mu$-measure and therefore contains $\operatorname{supp}\mu$.
\end{proof}

\section{The localized solid-torus construction}
\label{sec:localization}

The proof of \cref{thm:main} uses a two-stage local construction.  First, we
thicken a relative-boundary disk isotopy in auxiliary disk directions while
preserving its dynamics on a central slice.  Second, we couple the thickened
isotopy to a drift in a circle coordinate.  The resulting circle-coordinate
Lyapunov function has an equality set that is the disjoint union of the
prescribed invariant set and a set of fixed points.  This structure yields
both the isolation of the prescribed set and the invariant-measure
classification.  Since the entire construction is the identity on a boundary
collar, the resulting solid-torus diffeomorphism extends smoothly by the
identity to the ambient manifold.

Let $P$ be a compact smooth manifold with boundary.  A smooth isotopy
$(I_s)_{0\leq s\leq1}\subset\operatorname{Diff}^\infty(P)$ is
\emph{relative to the boundary} if there is a neighborhood $C$ of
$\partial P$ such that
\[
 I_s|_C=\id_C\qquad(0\leq s\leq1).
\]

The following elementary thickening keeps the original dynamics on a central
slice while making the map trivial near all newly introduced boundary
directions.

\begin{lemma}
\label[lemma]{lem:relative-boundary-thickening}
Let $m\geq1$ and $r\geq0$ be integers.  Let $D$ be a closed smooth $m$-disk
and let $(I_s)_{0\leq s\leq1}$ be a smooth isotopy of $D$ relative to the
boundary such that $I_0=\id_D$.  There are a closed smooth $(m+r)$-disk $D^\sharp$, a
smooth embedding $\kappa:D\to D^\sharp$ with
$\kappa(\operatorname{int}D)\subset\operatorname{int}D^\sharp$, and a
smooth relative-boundary isotopy
$(\overline I_s)_{0\leq s\leq1}$ of $D^\sharp$ such that
\[
 \overline I_0=\id_{D^\sharp},
 \qquad
 \overline I_s\circ\kappa=\kappa\circ I_s
 \quad(0\leq s\leq1).
\]
Moreover, suppose that $(I_s^j)_{0\leq s\leq1}$ is a sequence of such
isotopies, all equal to the identity on one fixed neighborhood of
$\partial D$, and that the evaluation maps
\[
 (s,x)\longmapsto I_s^j(x)
\]
converge in $C^\infty([0,1]\times D,D)$ to $(s,x)\mapsto x$.  Then the same
$D^\sharp$ and $\kappa$ may be used for every $j$, and the corresponding
evaluation maps $(s,y)\mapsto\overline I_s^j(y)$ converge in
$C^\infty([0,1]\times D^\sharp,D^\sharp)$ to $(s,y)\mapsto y$.
\end{lemma}

The first assertion will be used below for \cref{thm:main}; the uniform
sequential assertion is reserved for the near-identity construction in
\cref{sec:near-identity}.

\begin{proof}
For $r=0$, take $D^\sharp=D$, $\kappa=\id_D$, and
$\overline I_s=I_s$.  Suppose henceforth that $r\geq1$.  By the
relative-boundary hypothesis, choose an open neighborhood
$U_\partial\subset D$ of $\partial D$ such that
\[
 I_t|_{U_\partial}=\id_{U_\partial}
 \qquad (0\leq t\leq1).
\]
Let $B^r$ be the closed unit ball, and choose
$\gamma\in C^\infty(B^r,[0,1])$ such that $\gamma\equiv1$ near $0$ and
$\gamma\equiv0$ near $\partial B^r$.  On the product $D\times B^r$, whose
boundary faces $\partial D\times B^r$ and $D\times\partial B^r$ meet along
the codimension-two corner $\partial D\times\partial B^r$, define
\[
 J_s(x,z)=\bigl(I_{s\gamma(z)}(x),z\bigr).
\]

Define
\[
 \mathcal I:D\times[0,1]\longrightarrow D\times[0,1],
 \qquad
 \mathcal I(x,t)=\bigl(I_t(x),t\bigr).
\]
Since every $I_t$ is bijective, so is $\mathcal I$.  It remains to prove that
its inverse is smooth.  Fix $(x_0,t_0)\in D\times[0,1]$.  If
$x_0\in U_\partial$, then $\mathcal I$ and its inverse are the identity on
$U_\partial\times[0,1]$, so the inverse is smooth there.  Otherwise
$x_0\in D\setminus U_\partial\subset\operatorname{int}D$.
Since $(I_t)_{0\leq t\leq1}$ is a smooth isotopy, its evaluation map
\[
 [0,1]\times D\longrightarrow D,
 \qquad
 (t,x)\longmapsto I_t(x),
\]
is smooth up to the endpoints $t=0,1$.  Hence the coordinate expression of
$\mathcal I$ extends near $(x_0,t_0)$ to a smooth map between open subsets of
$\mathbb R^{m+1}$.  Its derivative at $(x_0,t_0)$ is
\[
 D\mathcal I_{(x_0,t_0)}=
 \begin{pmatrix}
  D_{x_0}I_{t_0} & \left.\partial_t I_t(x_0)\right|_{t=t_0}\\
  0               & 1
 \end{pmatrix},
\]
which is invertible because $D_{x_0}I_{t_0}$ is invertible.  The Euclidean
inverse function theorem \cite[Theorem~C.34]{Lee2013} therefore gives a smooth
local inverse.  Since every $I_t$ is bijective, this local inverse agrees on
$D\times[0,1]$ with $(y,t)\mapsto\bigl(I_t^{-1}(y),t\bigr)$.  As
$(x_0,t_0)$ was arbitrary, $\mathcal I$ has the smooth global inverse
\[
 \mathcal I^{-1}(y,t)=\bigl(I_t^{-1}(y),t\bigr).
\]

Consequently
\[
 J_s^{-1}(y,z)=\bigl(I_{s\gamma(z)}^{-1}(y),z\bigr),
\]
and both $J_s$ and $J_s^{-1}$ depend smoothly on their variables and on $s$.
Moreover, $I_0=\id_D$ gives
\[
 J_0(x,z)=\bigl(I_0(x),z\bigr)=(x,z),
\]
so $(J_s)_{0\leq s\leq1}$ is a smooth isotopy starting at the identity.

It is the identity on a
neighborhood of the entire boundary of $D\times B^r$.  Indeed, choose an open
neighborhood $V_\partial\subset B^r$ of $\partial B^r$ on which $\gamma=0$.
If $(x,z)\in U_\partial\times B^r$, then
\[
 J_s(x,z)=\bigl(I_{s\gamma(z)}(x),z\bigr)=(x,z)
\]
by the choice of $U_\partial$.  If $(x,z)\in D\times V_\partial$, then
\[
 J_s(x,z)=\bigl(I_0(x),z\bigr)=(x,z).
\]
Thus every $J_s$ is the identity on
\[
 (U_\partial\times B^r)\cup(D\times V_\partial),
\]
which is a neighborhood of
\[
 \partial(D\times B^r)
   =(\partial D\times B^r)\cup(D\times\partial B^r).
\]

Put
\[
 C_I=\overline{\{x\in D:I_t(x)\ne x
                     \text{ for some }t\in[0,1]\}},
\]
and, for brevity, call the set
\[
 C_J=\overline{\bigl\{(x,z)\in D\times B^r:
          J_s(x,z)\ne(x,z)\ \text{for some }s\in[0,1]\bigr\}}
\]
the joint support of $(J_s)_{0\leq s\leq1}$.
If $J_s(x,z)\ne(x,z)$ for some $s$, then
$I_{s\gamma(z)}(x)\ne x$ and, since $I_0=\id_D$, necessarily
$\gamma(z)\ne0$.  Thus $x\in C_I$ and
$z\in\operatorname{supp}\gamma$, and hence
\[
 C_J\subset C_I\times\operatorname{supp}\gamma.
\]
The set on the right is a compact subset of
$\operatorname{int}(D\times B^r)$: indeed,
$C_I\subset D\setminus U_\partial\subset\operatorname{int}D$, while
$\operatorname{supp}\gamma\subset B^r\setminus V_\partial
\subset\operatorname{int}B^r$.

The corner
$\Sigma=\partial D\times\partial B^r$ is disjoint from both
$C_I\times\operatorname{supp}\gamma$ and the central
slice $D\times\{0\}$.  Every $J_s$ is the identity on the open neighborhood
\[
 U_J=(U_\partial\times B^r)\cup(D\times V_\partial)
\]
of $\Sigma$.  Choose an open neighborhood $W$ of $\Sigma$ whose closure is
contained in $U_J$ and is disjoint from both sets.

Choose inward collar embeddings
\[
 c_D:\partial D\times[0,\varepsilon)\longrightarrow D,
 \qquad
 c_B:\partial B^r\times[0,\varepsilon)\longrightarrow B^r,
\]
with $c_D(x,0)=x$ and $c_B(z,0)=z$.  For $p=(x,z)\in\Sigma$, hold $x$ and
$z$ fixed and vary only the collar parameters:
\[
 (u,v)\longmapsto\bigl(c_D(x,u),c_B(z,v)\bigr).
\]
This gives a two-dimensional normal slice identified with the quadrant
$[0,\varepsilon)^2$.  In this slice, $u=0$ corresponds to
$\partial D\times B^r$, while $v=0$ corresponds to
$D\times\partial B^r$; the two boundary faces therefore meet at
$(u,v)=(0,0)$.  Choose a smooth rounded arc
$\Gamma\subset[0,\varepsilon)^2$ that cuts off a sufficiently small
neighborhood of $(0,0)$ and joins the remaining portions of the boundary edges
$\{u=0\}$ and $\{v=0\}$.  Choose $\Gamma$ to coincide with $\{u=0\}$ near one
endpoint and with $\{v=0\}$ near the other.  Since $\Sigma$ is compact and $W$
is a neighborhood of $\Sigma$, the arc may be chosen sufficiently close to
$(0,0)$ that using it in every normal slice modifies the boundary only inside
$W$.

In each normal slice, remove from $D\times B^r$ the small region between
$\Gamma$ and the corner $(0,0)$, retaining $\Gamma$ as the new boundary, and
leave $D\times B^r$ unchanged outside the product collar.  Denote the resulting
closed subset by $D^\sharp$.  Away from $W$, its boundary agrees with the
unchanged smooth portions of the original boundary; within the rounded collar,
the new part is $\Sigma\times\Gamma$.  This product is a smooth hypersurface,
and the endpoint conditions on $\Gamma$ make it join smoothly to the unchanged
portions of the two boundary faces.  Thus $\partial D^\sharp$ is smooth, so
$D^\sharp$ is a compact smooth domain contained in $D\times B^r$.

The passage
from $D\times B^r$ to $D^\sharp$ removes only points lying in $W$.  By the
choice of $W$,
\[
 W\cap\bigl(C_I\times\operatorname{supp}\gamma\bigr)=\varnothing,
 \qquad
 W\cap\bigl(D\times\{0\}\bigr)=\varnothing.
\]
Thus no point of either set is removed, and consequently
\[
 \bigl(C_I\times\operatorname{supp}\gamma\bigr)
   \cup\bigl(D\times\{0\}\bigr)\subset D^\sharp.
\]
Both $W$ and the rounded arc $\Gamma$ used in every normal slice are chosen
once, independently of $s$; hence the resulting domain $D^\sharp$ is fixed
throughout the isotopy.
Finally, the $\Gamma$-rounding described above turns $D\times B^r$, the
product of an $m$-disk and an $r$-disk, into the smooth domain $D^\sharp$,
which is diffeomorphic to the closed $(m+r)$-disk.

Since
$(D\times B^r)\setminus D^\sharp\subset W$, every $J_s$ fixes this complement
pointwise.  As $J_s$ is a bijection of $D\times B^r$, it follows that
$J_s(D^\sharp)=D^\sharp$.  Moreover, the rounded part of
$\partial D^\sharp$ lies in $W\subset U_J$, while its unchanged part lies in
$\partial(D\times B^r)\subset U_J$.  Thus $\partial D^\sharp\subset U_J$.
Since every $J_s$ is the identity on $U_J$, its restriction $\overline I_s$
is the identity on the neighborhood $U_J\cap D^\sharp$ of
$\partial D^\sharp$.  With $\kappa(x)=(x,0)$ one has
\[
 \overline I_s(\kappa(x))
   =\bigl(I_{s\gamma(0)}(x),0\bigr)
   =\kappa(I_s(x)).
\]

If a sequence of isotopies is given, let $C$ be their common fixed open
boundary collar and choose a smaller collar $C'$ with
$\overline{C'}\subset C$.  For each $j$, let $C_{I^j}$ denote the joint support
of $(I_s^j)$.  Then $C_{I^j}\subset D\setminus C'$.  Hence the joint support of
every corresponding family $(J_s^j)$ is contained in the fixed compact set
\[
 K_0=(D\setminus C')\times\operatorname{supp}\gamma,
\]
which is contained in $\operatorname{int}(D\times B^r)$.  With the fixed
neighborhood $V_\partial$ chosen above, put
\[
 U_*=(C\times B^r)\cup(D\times V_\partial),
 \qquad
 K_*=K_0\cup(D\times\{0\}).
\]
Every $J_s^j$ is the identity on $U_*$, while the compact set $K_*$ is
disjoint from $\Sigma=\partial D\times\partial B^r$.  Choose an open
neighborhood $W_*$ of $\Sigma$ such that
\[
 \overline{W_*}\subset U_*,
 \qquad
 \overline{W_*}\cap K_*=\varnothing.
\]
Perform the preceding $\Gamma$-rounding once, using fixed collar coordinates
and a fixed rounded arc, with the modification supported in $W_*$.  The
resulting domain $D^\sharp$ and the embedding $\kappa(x)=(x,0)$ are independent
of $j$.  Since every $J_s^j$ is the identity on $W_*$, the same
complement-fixing argument gives $J_s^j(D^\sharp)=D^\sharp$.

To verify the convergence assertion, write
\[
 \mathcal F_j:[0,1]\times D\longrightarrow D,
 \qquad
 \mathcal F_j(t,x)=I_t^j(x),
\]
and let $\mathcal F_\infty(t,x)=x$.  By hypothesis,
\[
 \mathcal F_j\longrightarrow\mathcal F_\infty
 \quad\text{in }C^\infty([0,1]\times D,D).
\]
Define the fixed smooth map
\[
 \Phi:[0,1]\times D\times B^r\longrightarrow[0,1]\times D,
 \qquad
 \Phi(s,x,z)=(s\gamma(z),x).
\]
The evaluation map of the thickened isotopy is
\[
 \mathcal J_j(s,x,z):=J_s^j(x,z)
   =\bigl((\mathcal F_j\circ\Phi)(s,x,z),z\bigr).
\]
Continuity of precomposition by the fixed smooth map $\Phi$ in the
$C^\infty$ topology therefore gives
\[
 \mathcal J_j\longrightarrow
 \bigl((s,x,z)\longmapsto(x,z)\bigr)
\]
in $C^\infty([0,1]\times D\times B^r,D\times B^r)$.

Let
\[
 \overline{\mathcal J}_j:[0,1]\times D^\sharp\longrightarrow D^\sharp,
 \qquad
 \overline{\mathcal J}_j(s,y)=\overline I_s^j(y).
\]
Since $D^\sharp$ is independent of $j$ and every $J_s^j$ preserves
$D^\sharp$, all these evaluation maps have the same domain and target.
Write $\iota:D^\sharp\hookrightarrow D\times B^r$ for the fixed smooth
inclusion.  Then
\[
 \iota\circ\overline{\mathcal J}_j
 =\mathcal J_j\circ(\id_{[0,1]}\times\iota).
\]
Restriction to the fixed smooth subdomain $[0,1]\times D^\sharp$ is
continuous in the $C^\infty$ topology, and the intrinsic $C^\infty$ topology
on maps into $D^\sharp$ is the topology induced by the fixed embedding
$\iota$.  Hence the preceding ambient convergence yields
\[
 \overline{\mathcal J}_j\longrightarrow
 \bigl((s,y)\longmapsto y\bigr)
 \quad\text{in }C^\infty([0,1]\times D^\sharp,D^\sharp).\qedhere
\]
\end{proof}

The following equality-set criterion isolates a distinguished invariant set
once a compact neighborhood separates it from the remainder of the equality
set.  It will supply the isolation step in
\cref{prop:localized-isolation}.

\begin{lemma}
\label[lemma]{lem:local-equality-isolation}
Let $T:Z\to Z$ be a homeomorphism of a compact metric space and let
$A:Z\to\mathbb R$ be continuous with $A\circ T\geq A$.  Put
\[
 Z_0=\{z\in Z:A(Tz)=A(z)\}.
\]
Suppose that $C\subset Z_0$ is compact and $T$-invariant, that
$A|_C\equiv c$, and that there is a compact set $W\subset Z$ such that
\[
 C\subset\operatorname{int}_Z W,
 \qquad Z_0\cap W=C.
\]
Then
\[
 C=\bigcap_{n\in\mathbb Z}T^n(W),
\]
and in particular $C$ is an isolated invariant set.
\end{lemma}

\begin{proof}
Only the reverse inclusion requires proof.  Suppose that the full orbit of
$z$ is contained in $W$ and put $a_n=A(T^nz)$.  The sequence $(a_n)_{n\in
\mathbb Z}$ is nondecreasing and bounded, so it has limits $a_-$ and $a_+$ at
$-\infty$ and $+\infty$.  For $N\geq1$,
\[
 \begin{aligned}
 \sum_{n=0}^{N-1}(a_{n+1}-a_n)
     &=a_N-a_0\longrightarrow a_+-a_0,\\
 \sum_{n=-N}^{-1}(a_{n+1}-a_n)
     &=a_0-a_{-N}\longrightarrow a_0-a_-.
 \end{aligned}
\]
Hence $a_{n+1}-a_n\to0$ as $n\to\pm\infty$.

By compactness, after passing
to subsequences,
\[
 T^{n_k}z\longrightarrow z_+\in W,
 \qquad
 T^{m_k}z\longrightarrow z_-\in W
\]
for some $n_k\to+\infty$ and $m_k\to-\infty$.  Since $A\circ T-A$ is
continuous, the preceding limit gives $z_+,z_-\in Z_0\cap W=C$.
Continuity of $A$ then yields
\[
 a_+=A(z_+)=c=A(z_-)=a_-.
\]
Thus $(a_n)$ is constant.  Every point $T^nz$ then belongs to
$Z_0\cap W=C$, proving the assertion.

Since
$C\subset\operatorname{int}_Z W$ and $C$ is $T$-invariant,
\[
 C\subset \bigcap_{n\in\mathbb Z}T^n(\operatorname{int}_Z W)
 \subset \bigcap_{n\in\mathbb Z}T^n(W)=C.
\]
Thus $\operatorname{int}_Z W$ is an open isolating neighborhood of $C$.
\end{proof}

\begin{proposition}
\label[proposition]{prop:localized-isolation}
Let $D$ be a closed smooth $m$-disk, where $m\geq2$, and let
$f:D\to D$ be a $C^\infty$ diffeomorphism joined to $\id_D$ by a smooth
isotopy relative to the boundary.  Let $\varnothing\ne X\subset\operatorname{int}D$
be compact and $f$-invariant, and suppose that $f|_X$ has no fixed point.  There
are a $C^\infty$ diffeomorphism $H$ of $D\times S^1$, equal to the identity on
a neighborhood of $\partial D\times S^1$, and an invariant set
\[
 K=X\times\{\pi/2\}
\]
with the following properties.  Writing $\iota_K(x)=(x,\pi/2)$, one has
\[
 (H|_K)\circ\iota_K=\iota_K\circ(f|_X),
 \qquad K\cap\operatorname{Fix}(H)=\varnothing,
\]
the set $K$ is isolated, and every $H$-invariant probability is supported on
$\operatorname{Fix}(H)\cup K$.  Moreover,
\begin{equation}
\label{eq:local-ergodic-measures}
 \Me(H)=\{\delta_z:z\in\operatorname{Fix}(H)\}
       \cup(\iota_K)_*\Me(f|_X).
\end{equation}
\begin{equation}
\label{eq:local-entropies}
 \He(H)=\{0\}\cup\He(f|_X),
 \qquad \htop(H)=\htop(f|_X).
\end{equation}
Finally, there are $C^\infty$ diffeomorphisms $H_n$ of $D\times S^1$, all
equal to the identity near the boundary, such that
\begin{equation}
\label{eq:local-approximation}
 H_n\longrightarrow H\quad\text{in }C^\infty,
 \qquad \htop(H_n)=0.
\end{equation}
The map $H$ and every $H_n$ are smoothly isotopic to the identity relative to
the same boundary collar.
\end{proposition}

\begin{proof}
We first construct $H=R\circ F$ and compute the equality set of the
circle-coordinate Lyapunov function.  We then derive the invariant-measure and
entropy classification and isolate $K$, and finally perturb the circle drift
to obtain $H_n$ and the relative-boundary isotopies.

Let $(I_s)_{0\leq s\leq1}$ be a relative-boundary isotopy from $\id_D$ to
$f$.  Choose a smooth nondecreasing map $\beta:[0,1]\to[0,1]$ which is
identically $0$ near $0$ and identically $1$ near $1$, and replace $I_s$ by
$I_{\beta(s)}$.  Thus we may assume that
\[
 I_s=\id_D\quad\text{for $s$ near $0$},
 \qquad
 I_s=f\quad\text{for $s$ near $1$}.
\]

The joint support
\[
 C_I=\overline{\{x\in D:I_s(x)\ne x
                    \text{ for some }s\in[0,1]\}}
\]
is a compact subset of $\operatorname{int}D$.  Choose
$b\in C^\infty(D,[0,1])$ which equals $1$ on a neighborhood of
$C_I\cup X$ and vanishes on a neighborhood of $\partial D$.
Let $\widehat D=D_+\cup_{\partial D}D_-$ be the smooth double of $D$,
obtained by gluing two copies of $D$ along their boundary; thus
$\widehat D\cong S^m$.  Identify $D$ with $D_+$.  Since
$X\subset\operatorname{int}D$ is closed in $\widehat D$,
\cref{lem:zero-set} gives
$\widehat\rho\in C^\infty(\widehat D,[0,\infty))$ such that
$\widehat\rho^{-1}(0)=X$.  Set $\rho=\widehat\rho|_D$; then
\[
 \rho^{-1}(0)=X.
\]
Put $M_\rho=\max_D\rho$ and fix
\[
 0<\varepsilon<(M_\rho+4)^{-1}.
\]
Write
\[
 p:\mathbb R\longrightarrow S^1=\mathbb R/(2\pi\mathbb Z)
\]
for the quotient map.

For $\theta\in S^1$, set
\[
 g_\theta=I_{\sin^2\theta},
 \qquad
 F(x,\theta)=(g_\theta(x),\theta).
\]
The map $F$ is a global $C^\infty$ diffeomorphism.  Indeed, the endpoint
stationarity lets us extend $(I_s)$ constantly to a smooth family on an open
interval $J\supset[0,1]$.  The constant extension retains the
relative-boundary property.  Since $C_I\subset\operatorname{int}D$ and $b$
vanishes near $\partial D$, fix a neighborhood $U_\partial\subset D$ of
$\partial D$ on which every $I_s$, $s\in J$, is the identity and
$b|_{U_\partial}=0$.

Define the map
\[
 \mathcal I:D\times J\longrightarrow D\times J,
 \qquad
 \mathcal I(x,s)=(I_s(x),s).
\]
It is bijective and has differential
\[
 D\mathcal I_{(x,s)}=
 \begin{pmatrix}
   D(I_s)_x & \partial_s I_s(x)\\
   0        & 1
\end{pmatrix},
\]
which is everywhere invertible.  Since $\mathcal I$ is the identity on
$U_\partial\times J$, define
\[
 \widetilde{\mathcal I}:\widehat D\times J
     \longrightarrow\widehat D\times J
\]
by
\[
 \widetilde{\mathcal I}|_{D_+\times J}=\mathcal I,
 \qquad
 \widetilde{\mathcal I}|_{D_-\times J}=\id_{D_-\times J}.
\]
This piecewise definition is smooth across $\partial D\times J$, because both
maps are the identity on a neighborhood of that set.  The map
$\widetilde{\mathcal I}$ is bijective, and its differential is everywhere
invertible by the displayed formula on $D_+\times J$ and by the identity
formula on $D_-\times J$.  Thus, by the inverse function theorem for manifolds
\cite[Theorem~4.5]{Lee2013}, $\widetilde{\mathcal I}$ is a $C^\infty$
diffeomorphism.  Restricting its inverse to $D_+\times J$ gives the
smooth inverse
\[
 \mathcal I^{-1}(y,s)=(I_s^{-1}(y),s).
\]
Consequently
\[
 F^{-1}(y,\theta)
   =\bigl(I_{\sin^2\theta}^{-1}(y),\theta\bigr)
\]
is smooth.

Using the cutoff $b$ chosen above, for $y\in D$ and $u\in\mathbb R$, put
\[
 a_b(y,u)=b(y)\sin u\bigl(\rho(y)+1-\sin u\bigr),
 \qquad
 \widetilde r_y(u)=u+\varepsilon a_b(y,u).
\]
Since $a_b(y,\cdot)$ is $2\pi$-periodic,
\[
 \widetilde r_y(u+2\pi)=\widetilde r_y(u)+2\pi.
\]
Moreover,
\begin{equation}
\label{eq:local-fiber-derivative}
 \partial_u\widetilde r_y(u)
 =1+\varepsilon b(y)\cos u
       \bigl(\rho(y)+1-2\sin u\bigr)
 \geq1-\varepsilon(M_\rho+3)>0.
\end{equation}
Also,
\[
 |\widetilde r_y(u)-u|\leq\varepsilon(M_\rho+2),
\]
so $\widetilde r_y(u)\to\pm\infty$ as $u\to\pm\infty$.  Thus every
$\widetilde r_y$ is an increasing diffeomorphism of $\mathbb R$, and it
induces an orientation-preserving circle diffeomorphism $r_y$ by
\[
 r_y(p(u))=p(\widetilde r_y(u)).
\]

To verify smooth dependence of the inverse on $y$, consider
\[
 \mathcal R:D\times\mathbb R\longrightarrow D\times\mathbb R,
 \qquad
 \mathcal R(y,u)=(y,\widetilde r_y(u)).
\]
This map is bijective, and its differential is block triangular with diagonal
blocks $\id_{T_yD}$ and the positive scalar in
\eqref{eq:local-fiber-derivative}.  Since $b=0$ on $U_\partial$, the map
$\mathcal R$ is the identity on $U_\partial\times\mathbb R$.  Thus
$\mathcal R$ extends by the identity to a smooth bijection of
$\widehat D\times\mathbb R$ whose differential is everywhere invertible.
Applying the inverse function theorem to this extension shows that
$\mathcal R^{-1}$ is smooth.  Let $\tau(y,u)=(y,u+2\pi)$.  The relation above
gives
\[
 \mathcal R\circ\tau=\tau\circ\mathcal R,
\]
and hence also
\[
 \mathcal R^{-1}\circ\tau=\tau\circ\mathcal R^{-1}.
\]
Thus both maps descend through the covering $\id_D\times p$ to mutually
inverse $C^\infty$ maps.  Consequently,
\[
 R(y,p(u))=\bigl(y,p(\widetilde r_y(u))\bigr)
\]
is a global $C^\infty$ diffeomorphism of $D\times S^1$.  Therefore
\[
H=R\circ F
\]
is a $C^\infty$ diffeomorphism.  On the fixed neighborhood $U_\partial$
chosen above, every $I_s$ is the identity and $b=0$.  Hence $F$, $R$, and $H$
are all the identity on $U_\partial\times S^1$.

Define
\[
 L:D\times S^1\longrightarrow\mathbb R,
 \qquad
 L(x,\theta)=-\cos\theta.
\]
Since $F$ preserves the circle coordinate, $L\circ F=L$.  We first determine
the increment of $L$ under $R$.  If $0<u<\pi$, then
\[
 0\leq\varepsilon a_b(y,u)
 \leq\varepsilon(M_\rho+1)\sin u
 <\pi-u,
\]
where we used $\sin u\leq\pi-u$.  Hence
\[
 u\leq\widetilde r_y(u)<\pi.
\]
Because $-\cos$ is strictly increasing on $(0,\pi)$,
\[
 -\cos\widetilde r_y(u)\geq-\cos u,
\]
with equality precisely when
\[
 b(y)=0
 \quad\text{or}\quad
 \bigl(\rho(y)=0\ \text{and}\ u=\pi/2\bigr).
\]
If instead $\pi<u<2\pi$, then
\[
 0\leq-\varepsilon a_b(y,u)
 \leq\varepsilon(M_\rho+2)(-\sin u)
 <u-\pi,
\]
and therefore
\[
 \pi<\widetilde r_y(u)\leq u.
\]
Since $-\cos$ is strictly decreasing on $(\pi,2\pi)$, its value again does
not decrease, and equality holds exactly when $b(y)=0$.  At $u=0$ and
$u=\pi$, one has $a_b(y,u)=0$.  It follows that
\[
 L\circ R-L\geq0
\]
and that its zero set is
\[
 \bigl(b^{-1}(0)\times S^1\bigr)
 \cup\bigl(D\times\{0,\pi\}\bigr)
 \cup\bigl(X\times\{\pi/2\}\bigr).
\]

We pull this equality set back by $F$.  Clearly
\[
 F^{-1}\bigl(D\times\{0,\pi\}\bigr)
   =D\times\{0,\pi\},
\]
because $g_0=g_\pi=I_0=\id_D$, and
\[
 F^{-1}\bigl(X\times\{\pi/2\}\bigr)
   =X\times\{\pi/2\}=K,
\]
because $g_{\pi/2}=I_1=f$ and $f^{-1}(X)=X$.

It remains to consider $b^{-1}(0)$.  Since $b\equiv1$ on a neighborhood
of $C_I$,
\[
 b^{-1}(0)\subset D\setminus C_I.
\]
By the definition of $C_I$, every point of $D\setminus C_I$ is fixed by
every $I_s$.  Hence $I_s$ fixes $b^{-1}(0)$ pointwise, and injectivity gives
\[
 I_s^{-1}\bigl(b^{-1}(0)\bigr)=b^{-1}(0)
 \qquad(s\in[0,1]).
\]
Consequently,
\[
 F^{-1}\bigl(b^{-1}(0)\times S^1\bigr)
   =b^{-1}(0)\times S^1.
\]
Since $L\circ F=L$, we have proved
\begin{equation}
\label{eq:local-equality-set}
 L\circ H-L\geq0,
 \qquad
 \{L\circ H=L\}=P\cup K,
\end{equation}
where
\[
 P=\bigl(b^{-1}(0)\times S^1\bigr)
   \cup\bigl(D\times\{0,\pi\}\bigr).
\]

If $x\in b^{-1}(0)$, then $I_s(x)=x$ for every $s$ and
$a_b(x,u)=0$ for every $u$.  Hence both $F$ and $R$ fix
$\{x\}\times S^1$.  If $\theta\in\{0,\pi\}$, then
$g_\theta=I_0=\id_D$ and $a_b(x,\theta)=0$, so both maps fix
$D\times\{\theta\}$.  Therefore
\[
 H|_P=\id_P.
\]

Moreover, $b=1$ on $X$ and $\pi/2\notin\{0,\pi\}$, so
$P\cap K=\varnothing$.  For $x\in X$, since $\sin^2(\pi/2)=1$ and
$I_1=f$,
\[
 F(x,\pi/2)=(I_1(x),\pi/2)=(f(x),\pi/2).
\]
Moreover, $f(x)\in X=\rho^{-1}(0)$, and hence
\[
 a_b\bigl(f(x),\pi/2\bigr)=b(f(x))\rho(f(x))=0.
\]
Therefore $R$ fixes $(f(x),\pi/2)$, and consequently
\[
 H(x,\pi/2)
   =R\bigl(F(x,\pi/2)\bigr)
   =(f(x),\pi/2).
\]
Thus $H|_K$ is conjugate to $f|_X$ through $\iota_K$ and therefore has no
fixed point.  Every fixed point of $H$ belongs to
$\{L\circ H=L\}=P\cup K$.  Consequently,
\[
 \operatorname{Fix}(H)=P.
\]

By \cref{lem:coboundary} and \eqref{eq:local-equality-set}, every invariant
probability is supported on $P\cup K$.  Since $P$ and $K$ are disjoint compact
invariant sets, an ergodic measure gives full mass to exactly one of them.
On $P$ the map is the identity, so its ergodic measures are precisely the
Dirac measures $\delta_z$, $z\in P=\operatorname{Fix}(H)$.  On $K$ the map
is conjugate to $f|_X$.  This proves \eqref{eq:local-ergodic-measures}.
Entropy is preserved under that conjugacy, while every fixed-point Dirac
measure has entropy zero; hence the first identity in
\eqref{eq:local-entropies} follows.  The ergodic variational principle
\cite{Walters1982} gives
\[
 \htop(H)=\htop(f|_X),
\]
which proves the second identity in \eqref{eq:local-entropies}.

To isolate $K$, choose a compact neighborhood $W$ of $K$ whose interior
contains $K$ and which is disjoint from $P$.  By
\eqref{eq:local-equality-set},
\[
 W\cap\{L\circ H=L\}=K,
 \qquad L|_K\equiv0.
\]
Applying \cref{lem:local-equality-isolation} with $T=H$, $A=L$, and $C=K$
proves that $K$ is isolated.

Finally, let $t_n\downarrow0$ with $0<t_n<1$ and put
\[
 a_{n,b}(y,u)
 =b(y)\sin u\bigl(\rho(y)+1+t_n-\sin u\bigr),
 \qquad
 \widetilde r_{n,y}(u)=u+\varepsilon a_{n,b}(y,u).
\]
Then
\[
 \partial_u\widetilde r_{n,y}(u)
 \geq1-\varepsilon(M_\rho+4)>0
\]
and the same bounded-displacement and inverse-function argument used above
shows that the $\widetilde r_{n,y}$ induce global smooth fiber
diffeomorphisms $R_n$ of $D\times S^1$.  Set $H_n=R_n\circ F$.  Since
$F|_{U_\partial\times S^1}=\id$ and $b|_{U_\partial}=0$, the definition
of $a_{n,b}$ gives $\widetilde r_{n,y}(u)=u$ for every
$y\in U_\partial$, $u\in\mathbb R$, and $n$.  Hence both $R_n$ and $H_n$
are the identity on the fixed collar $U_\partial\times S^1$ for every $n$.
Moreover,
\[
 a_{n,b}-a_b=t_n b(y)\sin u,
\]
so $R_n\to R$, and hence $H_n\to H$, in the $C^\infty$ topology.

The added term $t_n$ makes $\rho(y)+1+t_n-\sin u$ strictly positive.
Repeating the two semicircle estimates above therefore gives
\[
 L\circ H_n-L\geq0,\qquad
 \{L\circ H_n=L\}=P=\operatorname{Fix}(H_n),
 \qquad H_n|_P=\id_P.
\]
By \cref{lem:coboundary}, every $H_n$-invariant probability is supported on
$P$.  Thus every invariant measure has zero metric entropy, and the
variational principle gives $\htop(H_n)=0$.  This proves
\eqref{eq:local-approximation}.

For the isotopy assertion, let $0\leq\lambda\leq1$ and put
\[
 \begin{aligned}
 F_\lambda(x,\theta)
   &=\bigl(I_{\lambda\sin^2\theta}(x),\theta\bigr),\\
 \widetilde r_{\lambda,y}(u)
   &=u+\lambda\varepsilon a_b(y,u),\\
 \widetilde r_{n,\lambda,y}(u)
   &=u+\lambda\varepsilon a_{n,b}(y,u).
 \end{aligned}
\]
The last two maps have derivatives bounded below by
$1-\varepsilon(M_\rho+3)>0$ and $1-\varepsilon(M_\rho+4)>0$, respectively.
Hence the preceding inverse-function arguments give smooth families of fiber
diffeomorphisms $R_\lambda$ and $R_{n,\lambda}$.  The paths
$R_\lambda\circ F_\lambda$ and $R_{n,\lambda}\circ F_\lambda$ join the
identity to $H$ and $H_n$, respectively, and all their members equal the
identity on the same boundary collar.
\end{proof}

We now combine the Toeplitz entropy realization and disk-supported horseshoe
embedding from Section~2 with the relative-boundary thickening and localized
solid-torus construction above, and then embed the resulting model into $M$
and extend it by the identity.

\begin{proof}[Proof of \cref{thm:main}]
Fix $M$, $\mathscr S$, and $\mathfrak e$ as in the theorem.  Take the minimal
Toeplitz subshift $(Y,\sigma)$ and the affine entropy-preserving homeomorphism
$\phi:\mathscr S\to\Minv(\sigma|_Y)$ from
\cref{lem:entropy-source}.  By \cref{lem:horseshoe-embedding}, there are a
$C^\infty$ diffeomorphism $f$ of $S^2$, an invariant Cantor set $X$, and a
conjugacy $\chi:Y\to X$, and $f$ has support in the interior of a closed disk
$D_0\subset S^2$ containing $X$ in its interior.

Choose a closed square $Q$ and a larger closed disk $D$ such that
\[
 \operatorname{supp}(f)\cup X\subset\operatorname{int}Q,
 \qquad Q\subset\operatorname{int}D_0,
 \qquad D_0\subset\operatorname{int}D.
\]
Since $\operatorname{supp}(f)\subset\operatorname{int}Q$, the map $f$ is the
identity on a neighborhood of $\partial Q$; injectivity then implies
$f(Q)=Q$.  After identifying $Q$ smoothly with the unit square, $f|_Q$
belongs to the relative-boundary diffeomorphism group considered by Smale.
Smale's Theorem~B and its smooth-parametric form, Theorem~4
\cite{Smale1959}, give a jointly smooth isotopy $(J_s)_{0\leq s\leq1}$ from
$\id_Q$ to $f|_Q$, with every $J_s$ equal to the identity near $\partial Q$.

In these coordinates, set $\Delta(s,x):=J_s(x)-x\in\mathbb R^2$.  For each
$s$, the map $\Delta(s,\cdot)$ vanishes
near $\partial Q$, and hence, for every multi-index $\alpha$,
\[
 \partial_x^\alpha\Delta(s,x)=0
 \qquad(s\in[0,1],\ x\in\partial Q).
\]
The joint smoothness of $(s,x)\mapsto J_s(x)$ makes $\Delta$ jointly smooth.
Differentiating the displayed boundary identities with respect to $s$ shows
that every mixed derivative $\partial_s^k\partial_x^\alpha\Delta(s,x)$
vanishes on $[0,1]\times\partial Q$.  Extend each $J_s$ by the identity on
$D\setminus Q$, and denote the resulting family by $(I_s)_{0\leq s\leq1}$.
The vanishing just proved makes this family jointly smooth in $s$ and $x$, and
the endpoints satisfy
\[
 I_0=\id_D,
 \qquad
 I_1=f|_D.
\]
Since $Q\subset\operatorname{int}D$, this extended isotopy is the
identity on one fixed neighborhood of $\partial D$.

Since $Y$ is an infinite
minimal system, it has no periodic
point; in particular $f|_X$ has no fixed point.  We may therefore apply
\cref{lem:relative-boundary-thickening} with $r=d-3$.  Let
$D^\sharp$ be the resulting $(d-1)$-disk, let
$\kappa:D\to D^\sharp$ be the central-slice embedding, let
$(\overline I_s)$ be the resulting isotopy, and write
\[
 \overline f=\overline I_1,
 \qquad
 \overline X=\kappa(X).
\]
Then $\overline X\subset\operatorname{int}D^\sharp$ is compact and
$\overline f$-invariant, and
\[
 (\overline f|_{\overline X})\circ\kappa
   =\kappa\circ(f|_X).
\]

In particular, $\overline f|_{\overline X}$ has no fixed point.  Apply
\cref{prop:localized-isolation} to obtain $H$, $K$, and $(H_n)$ on
$D^\sharp\times S^1$.  Choose a fixed collar
$U_\partial\subset D^\sharp$ of $\partial D^\sharp$
such that $H$ and every $H_n$ are the identity on
$U_\partial\times S^1$, as provided by that proposition.

Inside a coordinate ball of $M$, take a small standard circle in a
two-dimensional coordinate plane.  Its tubular neighborhood is
diffeomorphic to $D^{d-1}\times S^1$.  Since $D^\sharp$ is a
$(d-1)$-disk, there is therefore a smooth embedding
\[
 \jmath:D^\sharp\times S^1\longrightarrow M.
\]
Put $N=\jmath(D^\sharp\times S^1)$, conjugate
$H$ and $H_n$ by $\jmath$ on $N$.  Their conjugates are the identity on
$\jmath(U_\partial\times S^1)$, a fixed collar of $\partial N$ inside $N$.
Hence their extensions by the identity to $M\setminus N$, denoted by $h$ and
$h_n$, are $C^\infty$ diffeomorphisms.  They are the identity on
$M\setminus\operatorname{int}N$, and \eqref{eq:local-approximation} gives
$h_n\to h$ in
$C^\infty$.

The relative-boundary isotopies in \cref{prop:localized-isolation}, after
conjugation by $\jmath$ and extension by the identity on $M\setminus N$, show
that
\[
 h\in\operatorname{Diff}_0^\infty(M),
 \qquad
 h_n\in\operatorname{Diff}_0^\infty(M)\quad(n\geq1).
\]

Put $K_M=\jmath(K)$.  The restriction $h|_{K_M}$ is conjugate to
$f|_X$, and hence to $\sigma|_Y$, via
$\jmath\circ\iota_K\circ\kappa\circ\chi$.
Consequently, $K_M$ is a minimal Cantor set and
$K_M\cap\operatorname{Fix}(h)=\varnothing$.

By construction,
\[
 h(N)=N=h^{-1}(N).
\]
Indeed, $h|_N=\jmath\circ H\circ\jmath^{-1}$ maps $N$ onto itself, whereas
$h$ is the identity on $M\setminus\operatorname{int}N$.

Let $\mu$ be an $h$-invariant probability, and set
\[
 \mu_N:=\mu|_N,
 \qquad
 \mu_{\mathrm{ext}}:=\mu|_{M\setminus N},
\]
as finite Borel measures on their respective subspaces.  When discussing
their supports or their sum, we extend them by zero to $M$.  Since both $N$
and $M\setminus N$ are $h$-invariant, so are $\mu_N$ and
$\mu_{\mathrm{ext}}$.  If $a:=\mu(N)>0$, then
\[
 \nu:=\frac{1}{a}(\jmath^{-1})_*\mu_N
\]
is an $H$-invariant probability.  The support conclusion of
\cref{prop:localized-isolation}, together with
$\jmath(\operatorname{Fix}(H))=\operatorname{Fix}(h)\cap N$, gives
\[
 \operatorname{supp}\mu_N
 \subset
 \jmath\bigl(\operatorname{Fix}(H)\cup K\bigr)
 =
 \bigl(\operatorname{Fix}(h)\cap N\bigr)\cup K_M.
\]
This inclusion is trivial when $a=0$.  On the other hand,
\[
 \operatorname{supp}\mu_{\mathrm{ext}}
 \subset \overline{M\setminus N}
 =M\setminus\operatorname{int}N
 \subset\operatorname{Fix}(h),
\]
because $h$ is the identity on $M\setminus\operatorname{int}N$.  Since
$\mu=\mu_N+\mu_{\mathrm{ext}}$, it follows that
\[
 \operatorname{supp}\mu\subset\operatorname{Fix}(h)\cup K_M.
\]

If $\mu$ is ergodic, the invariance of $N$ implies
$\mu(N)\in\{0,1\}$.  If $\mu(N)=0$, then $\mu$ is an ergodic probability for
the identity map on $M\setminus N$, and hence $\mu=\delta_z$ for some
$z\in M\setminus N$.  If $\mu(N)=1$, then
$(\jmath^{-1})_*(\mu|_N)$ is an ergodic $H$-invariant probability, so
\eqref{eq:local-ergodic-measures} applies and its classification transfers to
$h$ by $\jmath$.
Thus
\begin{equation}
\label{eq:global-ergodic-measures}
 \Me(h)=\{\delta_z:z\in\operatorname{Fix}(h)\}
 \cup (\jmath\circ\iota_K\circ\kappa\circ\chi)_*\Me(\sigma|_Y).
\end{equation}
The local isolating neighborhood of $K$ lies in the interior of the solid
torus, so its image shows that $K_M$ is isolated in $M$.

Set $\psi=\jmath\circ\iota_K\circ\kappa\circ\chi$ and
\[
 \mathcal F_{K_M}
 =\{\mu\in\Minv(h):\mu(K_M)=1\},
 \qquad \Phi=\psi_*\circ\phi.
\]
Then $\Phi$ is an affine homeomorphism from $\mathscr S$ onto
$\mathcal F_{K_M}$ and preserves the entropy function.

This face is
closed because, for any compatible metric $d$ on $M$,
\[
 \mathcal F_{K_M}
 =\left\{\mu\in\Minv(h):
      \int d(z,K_M)\,\dd\mu(z)=0\right\}.
\]
It is a face because a nontrivial convex combination assigns full mass to
$K_M$ only when both of its terms do.  Combining
\eqref{eq:global-ergodic-measures} with \cref{lem:entropy-source} and
\eqref{eq:local-entropies} gives the asserted ergodic measures and entropy
spectrum and, by the ergodic variational principle,
$\htop(h)=\max_{p\in\mathscr S}\mathfrak e(p)$.  This proves all the assertions
in
parts~\textup{(i)}--\textup{(ii)}, with the set denoted by $K_M$ here taken as
the set $K$ in the theorem statement.

Finally, let $\mu$ be an ergodic $h_n$-invariant probability.  Since
$h_n(N)=N=h_n^{-1}(N)$, ergodicity gives $\mu(N)\in\{0,1\}$.  If
$\mu(N)=0$, then $\mu$ is an ergodic probability for the identity map on
$M\setminus N$, and hence $\mu=\delta_z$ for some $z\in M\setminus N$; in
particular, $h_\mu(h_n)=0$.  If $\mu(N)=1$, then
\[
 \nu=(\jmath^{-1})_*(\mu|_N)
\]
is an ergodic $H_n$-invariant probability.  Since
$h_n|_N=\jmath\circ H_n\circ\jmath^{-1}$, invariance of measure-theoretic
entropy under conjugacy and \eqref{eq:local-approximation} give
\[
 h_\mu(h_n)=h_\nu(H_n)
 \leq \htop(H_n)=0.
\]
Thus every ergodic $h_n$-invariant probability has zero entropy, and the
ergodic form of the variational principle gives $\htop(h_n)=0$.  This proves
part~\textup{(iii)}.
\end{proof}

We now prove \cref{cor:compact-spectra} by a direct choice of the simplex in
Theorem~1.1.

\begin{proof}[Proof of \cref{cor:compact-spectra}]
Let $\mathscr S=\mathcal P(A)$, with the weak$^{*}$ topology, and set
$\mathfrak e(\lambda)=\int_A a\,\dd\lambda(a)$.  Then $\mathscr S$ is a compact
metrizable Choquet simplex, $\mathfrak e$ is nonnegative, continuous, and
affine, and $\operatorname{ex}\mathscr S=\{\delta_a:a\in A\}$.  Thus
$\mathfrak e(\operatorname{ex}\mathscr S)=A$ and
$\max_{\mathscr S}\mathfrak e=\max A$.  Apply \cref{thm:main}(ii).
\end{proof}

\section{Counterexamples near the identity}
\label{sec:near-identity}

We first record three auxiliary facts.  Lemma~4.1 extends near-identity disk
diffeomorphisms relative to a fixed outer boundary, Lemma~4.2 extracts the
$1/N$ entropy scaling from primitive-renormalization data, and Lemma~4.3 makes
the localized implantation converge smoothly to the identity.

\begin{lemma}
\label[lemma]{lem:relative-extension}
Let $D\subset\operatorname{int}\widehat D$ be concentric closed disks in
$\mathbb R^2$.  If
$q_j\in\operatorname{Diff}^\infty(D)$ and
$q_j\to\id_D$ in the $C^\infty$ topology, then, after discarding finitely many
terms, there are
$\widehat q_j\in\operatorname{Diff}^\infty(\widehat D)$ such that
\[
 \begin{aligned}
 \widehat q_j|_D&=q_j,\\
 \widehat q_j&=\id_{\widehat D}
   &&\text{on one fixed neighborhood of }\partial\widehat D,\\
 \widehat q_j&\longrightarrow\id_{\widehat D}
   &&\text{in }C^\infty.
 \end{aligned}
\]
\end{lemma}

\begin{proof}
We first construct a continuous linear extension operator for smooth vector
fields on $D$, with image supported in a fixed compact subset of
$\operatorname{int}\widehat D$.  We then apply it to $q_j-\id_D$ and use
$C^1$-smallness to obtain global diffeomorphisms.

Since $D$ is a compact subset of $\operatorname{int}\widehat D$, choose
$\zeta\in C_c^\infty(\mathbb R^2)$ such that $0\leq\zeta\leq1$,
$\operatorname{supp}\zeta\subset\operatorname{int}\widehat D$, and
$\zeta\equiv1$ on a neighborhood of $D$.

Let
\[
 \mathbb H^+=\{(x_1,x_2)\in\mathbb R^2:x_2\geq0\}.
\]
We use $\mathbb H^+$ for the closed half-space.  The notation
$C^\infty(\mathbb H^+)$ means Seeley's space $D_+$: the functions smooth on
the open half-space $\{x_2>0\}$ whose derivatives of every order extend
continuously to $\{x_2=0\}$ \cite[p.~625]{Seeley1964}.

To obtain an extension on $D$ from Seeley's theorem, fix a finite open cover
$U_0,U_1,\ldots,U_m$ of $D$ such that
$\overline{U_0}\subset\operatorname{int}D$ and, for every $i\geq1$, there is
a smooth chart
\[
 \kappa_i:U_i\longrightarrow V_i\subset\mathbb R^2
\]
satisfying
\[
 \kappa_i(U_i\cap D)=V_i\cap\mathbb H^+.
\]
For $\rho\in C^\infty(D)$, write
\[
 \operatorname{supp}_D\rho
 :=\overline{\{x\in D:\rho(x)\neq0\}}
\]
for its support relative to $D$.
Choose a smooth partition of unity $(\rho_i)_{i=0}^m$ subordinate to this
cover, so that each $\operatorname{supp}_D\rho_i$ is a compact subset of
$U_i\cap D$.  For every $i\geq1$, choose
$\theta_i\in C_c^\infty(U_i)$ such that $0\leq\theta_i\leq1$ and
$\theta_i\equiv1$ on a neighborhood of $\operatorname{supp}_D\rho_i$.

For $v\in C^\infty(D,\mathbb R^2)$, write
\[
 (\rho_i\cdot v)(x):=\rho_i(x)v(x)
 \qquad (x\in D,\ 0\leq i\leq m).
\]
We define $E_0(v)$ as follows.
Since $\operatorname{supp}_D\rho_0\subset U_0\cap D$ and
$\overline{U_0}\subset\operatorname{int}D$, the function $\rho_0\cdot v$
vanishes on a neighborhood of $\partial D$.  It therefore extends by zero to
a function
\[
 \widetilde v_0\in C_c^\infty(\mathbb R^2,\mathbb R^2).
\]
The map $v\mapsto\widetilde v_0$ is continuous and linear.

Let
\[
 \mathcal S:C^\infty(\mathbb H^+,\mathbb R^2)
 \longrightarrow C^\infty(\mathbb R^2,\mathbb R^2)
\]
be the continuous linear operator obtained by applying Seeley's extension
operator componentwise.  By construction, extending a function and then
restricting it back to $\mathbb H^+$ recovers the original function:
\[
 \mathcal S(w)|_{\mathbb H^+}=w
 \qquad
 \bigl(w\in C^\infty(\mathbb H^+,\mathbb R^2)\bigr).
\]

For every $i\geq1$, define $g_i:\mathbb H^+\to\mathbb R^2$ by
\[
 g_i(z)=
 \begin{cases}
  ((\rho_i\cdot v)\circ\kappa_i^{-1})(z),
    & z\in V_i\cap\mathbb H^+,\\[2mm]
  0, & z\in\mathbb H^+\setminus V_i.
 \end{cases}
\]
Since $\operatorname{supp}_D\rho_i$ is a compact subset of $U_i\cap D$,
this zero extension belongs to $C^\infty(\mathbb H^+,\mathbb R^2)$.  Define
\[
 \widetilde v_i(x)=
 \begin{cases}
  \theta_i(x)\,\mathcal S(g_i)\bigl(\kappa_i(x)\bigr),
    & x\in U_i,\\[2mm]
  0, & x\notin U_i.
 \end{cases}
\]
Since $\theta_i$ has compact support in $U_i$, the two branches fit together
smoothly.  Hence $\widetilde v_i\in
C_c^\infty(\mathbb R^2,\mathbb R^2)$, the map
$v\mapsto\widetilde v_i$ is continuous and linear.

On $D\cap U_i$, the
preceding restriction identity for $\mathcal S$ and the choice of $\theta_i$
give $\widetilde v_i=\rho_i\cdot v$.  If $x\in D\setminus U_i$, then the
definition of $\widetilde v_i$ and the subordination of $\rho_i$ give
\[
 \widetilde v_i(x)=0=(\rho_i\cdot v)(x).
\]
Therefore
\[
 \widetilde v_i|_D=\rho_i\cdot v.
\]

Set
\[
 E_0(v)=\sum_{i=0}^m\widetilde v_i.
\]
The preceding construction shows that
\[
 E_0:C^\infty(D,\mathbb R^2)\longrightarrow
 C^\infty(\mathbb R^2,\mathbb R^2)
\]
is continuous and linear, and
\[
 E_0(v)|_D=\sum_{i=0}^m(\rho_i\cdot v)=v.
\]

Define $E(v)=\zeta\cdot E_0(v)$.  Multiplication by the fixed smooth function
$\zeta$ is a continuous linear endomorphism of
$C^\infty(\mathbb R^2,\mathbb R^2)$; hence
\[
 E:C^\infty(D,\mathbb R^2)\longrightarrow
 C^\infty(\mathbb R^2,\mathbb R^2)
\]
is continuous and linear.  Moreover, $E(v)$ agrees with $v$ on $D$ and is
supported in the fixed compact set
$K_\zeta:=\operatorname{supp}\zeta\subset\operatorname{int}\widehat D$.

For
$x\in D$, set
\[
 u_j(x):=q_j(x)-x\in\mathbb R^2.
\]
Then $u_j\in C^\infty(D,\mathbb R^2)$ and, since $q_j\to\id_D$, we have
$u_j\to0$ in $C^\infty(D,\mathbb R^2)$.  Define
\[
 Q_j:\mathbb R^2\longrightarrow\mathbb R^2,
 \qquad
 Q_j(x):=x+E(u_j)(x).
\]
Since $E(u_j)|_D=u_j$, for $x\in D$ we have
\[
 Q_j(x)=x+u_j(x)=q_j(x),
\]
and hence $Q_j|_D=q_j$.  Moreover,
$\operatorname{supp}E(u_j)\subset K_\zeta$, so $Q_j$ is the identity outside
$K_\zeta$.
Consequently $Q_j$ is proper: if $C\subset\mathbb R^2$ is compact, then
$Q_j^{-1}(C)$ is closed and contained in the compact set
$C\cup K_\zeta$, and hence is compact.

Since $u_j\to0$ and $E$ is continuous and
linear,
\[
 E(u_j)\longrightarrow E(0)=0
 \quad\text{in }C^\infty(\mathbb R^2,\mathbb R^2).
\]
Because every $E(u_j)$ is supported in the same compact set $K_\zeta$, this also
gives
\[
 \lVert D E(u_j)\rVert_{C^0(\mathbb R^2)}
 =\lVert D E(u_j)\rVert_{C^0(K_\zeta)}\longrightarrow0.
\]

Thus, for all large $j$,
$\lVert D E(u_j)\rVert_{C^0(\mathbb R^2)}<1/2$, and
\[
 DQ_j(x)=I+D E(u_j)(x)
\]
is everywhere invertible.  Hence $Q_j$ is a local diffeomorphism.  It is
injective by the mean-value estimate
\[
 \lVert Q_j(x)-Q_j(y)\rVert
 \geq\tfrac12\lVert x-y\rVert.
\]

Since $Q_j$ is a local diffeomorphism, its image is open in $\mathbb R^2$.
Properness implies that $Q_j$ is a closed map, so its image is also closed.
Since $\mathbb R^2$ is connected, the nonempty image must be all of
$\mathbb R^2$.
Together with the preceding injectivity, this shows that $Q_j$ is bijective.
A bijective local diffeomorphism has a smooth inverse; hence $Q_j$ is a global
diffeomorphism of $\mathbb R^2$.

Because
$K_\zeta\subset\operatorname{int}\widehat D$, the complement
$\mathbb R^2\setminus\widehat D$ is pointwise fixed; bijectivity then gives
$Q_j(\widehat D)=\widehat D$.  Set
\[
 U_{\partial}:=\widehat D\setminus K_\zeta.
\]
This is a fixed relative neighborhood of $\partial\widehat D$.  For
$\widehat q_j:=Q_j|_{\widehat D}$ we have
\[
 \widehat q_j|_D=q_j,
 \qquad
 \widehat q_j|_{U_{\partial}}=\id_{U_{\partial}}.
\]

Moreover, the definition of $Q_j$ gives, for every $x\in\widehat D$,
\[
 \widehat q_j(x)-x=Q_j(x)-x=E(u_j)(x).
\]
The $C^\infty$ convergence $E(u_j)\to0$ established above, restricted to
$\widehat D$, yields
\[
 \widehat q_j\longrightarrow\id_{\widehat D}
 \quad\text{in }C^\infty.
\]
This proves the lemma.
\end{proof}

\begin{lemma}
\label[lemma]{lem:cyclic-renormalization}
Let $G:D_0\to D_0$ and $q:D_1\to D_1$ be diffeomorphisms of closed disks.
Suppose that for some $N\geq2$ and some smooth embedding
$\psi:D_0\to D_1$,
\[
 q^N\circ\psi=\psi\circ G,
 \qquad
 q^i(\psi(D_0))\cap\psi(D_0)=\varnothing
 \quad(0<i<N).
\]
Let $Z\subset\operatorname{int}D_0$ be compact with $G(Z)=Z$ and such that
$G|_Z$ is minimal and uniquely ergodic, with unique invariant probability
$\mu$ and entropy $h_\mu(G)=H$.  Then
\[
 X=\bigcup_{i=0}^{N-1}q^i(\psi(Z))
\]
is a compact $q$-invariant set on which $q$ is minimal, uniquely ergodic, and
has no fixed point.  Its unique invariant probability is
\[
 \eta=\frac1N\sum_{i=0}^{N-1}(q^i\circ\psi)_*\mu,
\]
and
\[
 h_\eta(q)=\htop(q|_X)=\frac{H}{N}.
\]
If $Z$ is a Cantor set, then so is $X$.
\end{lemma}

\begin{proof}
Put $L_i=q^i(\psi(Z))$ for $0\leq i<N$, with indices read modulo $N$.
If $i<j$, then $L_i\cap L_j\ne\varnothing$ would, after applying $q^{-i}$,
contradict the hypothesis for $j-i$.  The conjugacy identity and $G(Z)=Z$
give $q(L_i)=L_{i+1}$, while $q^N|_{L_i}$ is conjugate to $G|_Z$.  Thus
$X=\bigsqcup_iL_i$ is compact and $q$-invariant, and every $q$-orbit is dense
in $X$.

Set $\nu_i=(q^i\circ\psi)_*\mu$.  Since $q_*\nu_i=\nu_{i+1}$ cyclically,
$\eta=N^{-1}\sum_i\nu_i$ is $q$-invariant.  Conversely, if $\rho$ is
$q$-invariant, cyclicity gives $\rho(L_i)=1/N$, while
$(\psi^{-1})_*(N\rho|_{L_0})$ is $G$-invariant and hence equals $\mu$.
Pushing this restriction around the tower gives $\rho=\eta$.

Each $(L_i,q^N,\nu_i)$ is conjugate to $(Z,G,\mu)$.  By affinity of entropy
on the finite $q^N$-invariant decomposition,
\[
 h_\eta(q^N)
 =\frac1N\sum_{i=0}^{N-1}h_{\nu_i}(q^N)
 =H.
\]
The iterate formula gives $h_\eta(q)=H/N$, and unique ergodicity plus the
variational principle gives $\htop(q|_X)=h_\eta(q)$.  Finally,
$q(L_i)=L_{i+1}\ne L_i$ excludes fixed points, and if $Z$ is Cantor, so is
$X=\bigsqcup_iL_i$.
\end{proof}

\begin{lemma}
\label[lemma]{lem:small-localized-implantation}
Let $D=\overline B(c,r)\subset\mathbb R^2$ be a closed disk, and let
$f_j\in\operatorname{Diff}^\infty(D)$ all be equal to the identity on one
fixed neighborhood of $\partial D$, and suppose that $f_j\to\id_D$ in
$C^\infty$.  Let the nonempty compact $f_j$-invariant sets
$X_j\subset\operatorname{int}D$ lie in one fixed compact subset of
$\operatorname{int}D$, and suppose that $f_j|_{X_j}$ has no fixed point.  For
every integer $r_0\geq0$, after discarding finitely many terms and
reindexing, there are a fixed closed smooth $(r_0+2)$-disk
$D^\sharp$, a fixed smooth embedding $\kappa:D\to D^\sharp$, and
diffeomorphisms $H_j$ of $D^\sharp\times S^1$, all equal to the identity
on one fixed boundary neighborhood, such that $H_j$ and
\[
 K_j=\kappa(X_j)\times\{\pi/2\}
\]
have all the conclusions of \cref{prop:localized-isolation}, with
$f_j|_{X_j}$ as the dynamics on $K_j$, and
\[
 H_j\longrightarrow\id_{D^\sharp\times S^1}
 \quad\text{in }C^\infty.
\]
The relative-boundary isotopies $(H_{j,\lambda})_{0\leq\lambda\leq1}$ from
the identity to $H_j$ may be chosen so that the evaluation maps
$(\lambda,z)\mapsto H_{j,\lambda}(z)$ converge jointly in $C^\infty$ to
$(\lambda,z)\mapsto z$.
\end{lemma}

\begin{proof}
Choose one endpoint-stationary smooth function
$\beta:[0,1]\to[0,1]$ which is identically zero near $0$ and identically one
near $1$.  For every $j$, define
\[
 I_s^j(x)=x+\beta(s)(f_j(x)-x)
 \qquad(0\leq s\leq1).
\]
Since $D$ is convex,
\[
 I_s^j(x)=(1-\beta(s))x+\beta(s)f_j(x)\in D.
\]
Since $f_j\to\id_D$ in $C^1$, after discarding finitely many terms and
reindexing one has, for every $j$,
\[
 D_xI_s^j=I+\beta(s)\bigl(Df_j(x)-I\bigr)
\]
invertible uniformly in $(s,x)$.  By hypothesis, $I_s^j$ is the identity on
one fixed neighborhood of $\partial D$, independently of $j$ and $s$.
Together with the preceding derivative estimate, this makes $I_s^j:D\to D$
a local diffeomorphism of manifolds with boundary; since $D$ is compact, it is
proper.  The
proper-local-diffeomorphism criterion
\cite[Proposition~4.46]{Lee2013} shows that it is a smooth covering of $D$.
Since the base $D$ is simply connected and the domain $D$ is connected, this
covering is bijective; hence $(I_s^j)_{0\leq s\leq1}$ is a
relative-boundary isotopy from $\id_D$ to $f_j$.  Moreover,
$\beta^{(0)}:=\beta$ and, for every multi-index $\alpha$ and every
$\ell\geq0$,
\[
 \partial_s^\ell\partial_x^\alpha\bigl(I_s^j(x)-x\bigr)
 =\beta^{(\ell)}(s)
   \partial_x^\alpha\bigl(f_j(x)-x\bigr).
\]
Therefore
$I_s^j\to\id_D$ jointly in $(s,x)$ in the $C^\infty$ topology.

Apply
\cref{lem:relative-boundary-thickening} with $r=r_0$.  It gives a fixed
$(r_0+2)$-disk $D^\sharp$, a fixed central-slice embedding
$\kappa:D\to D^\sharp$, and relative-boundary isotopies
$(\overline I_s^j)$ such that
\[
 \overline I_s^j\circ\kappa=\kappa\circ I_s^j,
 \qquad
 (s,y)\longmapsto\overline I_s^j(y)
   \longrightarrow (s,y)\longmapsto y
\]
in $C^\infty([0,1]\times D^\sharp,D^\sharp)$.  Put
\[
 \overline f_j=\overline I_1^j,
 \qquad
 \overline X_j=\kappa(X_j).
\]
After composing the parameter of every $\overline I_s^j$ with one fixed
smooth function that is $0$ near $0$ and $1$ near $1$, and retaining the
same notation, we may also assume that these isotopies are stationary near
both endpoints.  This reparametrization preserves their endpoints, the
central-slice identity, and the preceding joint convergence.
Then $\overline X_j$ is compact and $\overline f_j$-invariant,
$\overline f_j|_{\overline X_j}$ is conjugate to $f_j|_{X_j}$ through
$\kappa$, and all the sets $\overline X_j$ lie in one fixed compact subset
of $\operatorname{int}D^\sharp$.  The angle-dependent maps
\[
 F_j(y,\theta)=\bigl(\overline I^j_{\sin^2\theta}(y),\theta\bigr)
\]
used in the proof of \cref{prop:localized-isolation} therefore converge to
the identity.

The thickening construction provides a fixed neighborhood $U_0$ of
$\partial D^\sharp$ on which every $\overline I_s^j$ is the identity.  Let
$K_0\subset\operatorname{int}D^\sharp$ be a fixed compact set containing
every $\overline X_j$, and choose a fixed boundary collar $U_\partial$ such
that $\overline{U_\partial}\subset U_0$ and
$\overline{U_\partial}\cap K_0=\varnothing$.  Hence the joint support
\[
 C_j=\overline{\{y\in D^\sharp:\overline I_s^j(y)\ne y
                 \text{ for some }s\in[0,1]\}}
\]
and the set $\overline X_j$ are both disjoint from
$\overline{U_\partial}$.  Choose $b_j$ as in the proof of
\cref{prop:localized-isolation}, equal to $1$ near
$C_j\cup\overline X_j$ and vanishing on $U_\partial$.  Choose $\rho_j$ so
that $\rho_j^{-1}(0)=\overline X_j$, and write $R_{j,\varepsilon}$ for the
corresponding fiber drift with parameter $\varepsilon$.

For fixed $j$, one has $R_{j,\varepsilon}\to\id$ in $C^\infty$ as
$\varepsilon\downarrow0$.  More precisely, write
\[
 a_{j,b_j}(y,u)
 =b_j(y)\sin u\bigl(\rho_j(y)+1-\sin u\bigr).
\]
The evaluation map of the scaled fiber path is
\[
 \mathcal R_{j,\varepsilon}(\lambda,y,p(u))
  =\bigl(y,p(u+\lambda\varepsilon a_{j,b_j}(y,u))\bigr).
\]
For fixed $j$, as $\varepsilon\downarrow0$ these maps converge in
$C^\infty([0,1]\times D^\sharp\times S^1,D^\sharp\times S^1)$ to
$(\lambda,y,\theta)\mapsto(y,\theta)$, because on the lifts their difference
is $\lambda\varepsilon a_{j,b_j}(y,u)$ and all its mixed derivatives tend
uniformly to zero.  Hence we may choose
\[
 0<\varepsilon_j<(\max_{D^\sharp}\rho_j+4)^{-1}
\]
so small that the evaluation maps of
$(R_{j,\lambda\varepsilon_j})_{0\leq\lambda\leq1}$ are within $1/j$ of the
identity evaluation map in $C^j$.  All derivative and equality-set arguments
in \cref{prop:localized-isolation} remain valid for this choice.

For
$0\leq\lambda\leq1$, put
\[
 F_{j,\lambda}(y,\theta)
   =\bigl(\overline I^j_{\lambda\sin^2\theta}(y),\theta\bigr),
 \qquad
 H_{j,\lambda}
   =R_{j,\lambda\varepsilon_j}\circ F_{j,\lambda}.
\]
These are smooth relative-boundary isotopies, and
$H_j=H_{j,1}$ has all the conclusions of
\cref{prop:localized-isolation}.  The joint convergence of
$\overline I_s^j$ and the choice of $\varepsilon_j$ give
\[
 (\lambda,z)\longmapsto H_{j,\lambda}(z)
   \longrightarrow(\lambda,z)\longmapsto z
\]
in $C^\infty([0,1]\times(D^\sharp\times S^1),
D^\sharp\times S^1)$.  In particular,
$H_j\to\id_{D^\sharp\times S^1}$ in $C^\infty$.
\end{proof}

The preceding three lemmas provide the relative extension, cyclic entropy
scaling, and localized implantation needed for the near-identity construction.

\begin{proof}[Proof of \cref{thm:near-identity}]
Fix $H_*>0$.  Apply \cref{lem:entropy-source} to the one-point simplex with
entropy function equal to $H_*$.  This gives a minimal uniquely ergodic
Toeplitz subshift $(Y,\sigma)$ whose unique invariant probability has entropy
$H_*$.  Let the orientation-preserving $C^\infty$ diffeomorphism
$f:S^2\to S^2$, the set $X$, the conjugacy $\chi:Y\to X$, and the $k$-leg horseshoe
$\Lambda\supset X$ be supplied by
\cref{lem:horseshoe-embedding}.  By the final clause of that lemma, choose a
closed disk $E\subset S^2$ such that
\[
 \operatorname{supp}(f)\cup\Lambda\subset\operatorname{int}E.
\]

Then $f$ is the identity on a neighborhood of
$S^2\setminus\operatorname{int}E$.  In particular,
$f(S^2\setminus E)=S^2\setminus E$, so bijectivity gives $f(E)=E$.

Choose a smooth identification
\[
 \Theta:E\longrightarrow D=\overline B(c,r)\subset\mathbb R^2
\]
and set
\[
 G=\Theta\circ f|_E\circ\Theta^{-1},
 \qquad
 Z=\Theta(X).
\]
Then $G:D\to D$ is orientation preserving and equal to the identity near
$\partial D$, while $Z\subset\operatorname{int}D$ is a $G$-invariant Cantor
set and $\Theta\circ\chi$ conjugates $\sigma|_Y$ to $G|_Z$.  Finally, the
constant sequences in the full $k$-shift give hyperbolic fixed points in
$\Theta(\Lambda)$; fix one and denote it by $p$.

Let $(\mathcal U_j)$ be a decreasing countable neighborhood basis of $\id_D$
in $\operatorname{Diff}^\infty(D)$.  Put $B^2=\overline B(0,1)$, and let
\[
 A:B^2\longrightarrow D,
 \qquad
 A(x)=c+rx,
\]
be the orientation-preserving affine diffeomorphism.  Conjugation by $A$ is
a homeomorphism from $\operatorname{Diff}^\infty(B^2)$ to
$\operatorname{Diff}^\infty(D)$ in the $C^\infty$ topology.

Apply the
$C^\infty$ case of Theorem~A of Berger, Gourmelon, and Helfter on primitive
renormalizations
\cite[Definition~1.1 and Theorem~A]{BergerGourmelonHelfter2025} to
$A^{-1}\circ G\circ A$, with the neighborhood
\[
 A^{-1}\mathcal U_jA
 :=\{A^{-1}\circ q\circ A:q\in\mathcal U_j\},
\]
and conjugate the resulting data back by $A$.  We obtain
$q_j\in\mathcal U_j$, integers
$N_j\geq2$, and smooth embeddings $\psi_j:D\to D$ such that
\[
 q_j^{N_j}\circ\psi_j=\psi_j\circ G,
 \qquad
 q_j^i(\psi_j(D))\cap\psi_j(D)=\varnothing
 \quad(0<i<N_j).
\]
Since $(\mathcal U_j)$ is a decreasing neighborhood basis,
\[
 q_j\longrightarrow\id_D\quad\text{in }C^\infty.
\]

We claim that $N_j\to\infty$.  Otherwise, $(N_j)$ has a bounded subsequence;
after passing further, we may assume that $N_j=N$ is constant.  Put
$z_j=\psi_j(p)$.  Since $G(p)=p$, differentiating
$q_j^N\circ\psi_j=\psi_j\circ G$ at $p$ gives
\[
 D(q_j^N)_{z_j}
 =D(\psi_j)_p\,DG_p\,D(\psi_j)_p^{-1}.
\]
For fixed $N$, the $C^1$ continuity of the iteration map gives
$q_j^N\to\id_D$ in $C^1$.  Hence
\[
 \sup_{z\in D}\bigl\|D(q_j^N)_z-\operatorname{Id}\bigr\|
 \longrightarrow 0,
\]
and therefore $D(q_j^N)_{z_j}\to\operatorname{Id}$.  On the other hand, the
identity obtained by differentiating
$q_j^N\circ\psi_j=\psi_j\circ G$ at $p$ shows that every
$D(q_j^N)_{z_j}$ is similar to $DG_p$ and hence has the same characteristic
polynomial as $DG_p$.  Since the
characteristic polynomial depends continuously on the matrix entries,
passing to the limit would give
\[
 \det\bigl(t\operatorname{Id}-DG_p\bigr)=(t-1)^2,
\]
contradicting the hyperbolicity of $DG_p$.  Thus $N_j\to\infty$.

Choose a larger concentric closed disk $\widehat D$ such that
$D\subset\operatorname{int}\widehat D$.  Since $q_j\to\id_D$ in the
$C^\infty$ topology, \cref{lem:relative-extension} applies.  After discarding
finitely many terms and reindexing, let
$\widehat q_j\in\operatorname{Diff}^\infty(\widehat D)$ be the extensions
furnished by that lemma.  Thus
\[
 \widehat q_j|_D=q_j,
 \qquad
 \widehat q_j\longrightarrow\id_{\widehat D}\quad\text{in }C^\infty,
\]
and all $\widehat q_j$ equal $\id_{\widehat D}$ on one fixed neighborhood of
$\partial\widehat D$.  Since $q_j(D)=D$, one has
$\widehat q_j^i|_D=q_j^i$ for every $i\geq0$.  As $\psi_j(D)\subset D$, it
follows that
\[
 \widehat q_j^{N_j}\circ\psi_j=\psi_j\circ G,
 \qquad
 \widehat q_j^i\bigl(\psi_j(D)\bigr)\cap\psi_j(D)=\varnothing
 \quad(0<i<N_j).
\]

Define
\[
 X_j=\bigcup_{i=0}^{N_j-1}\widehat q_j^i(\psi_j(Z)),
 \qquad
 c_j=\frac{H_*}{N_j}.
\]
By \cref{lem:cyclic-renormalization}, $X_j$ is a minimal uniquely ergodic
invariant Cantor set without fixed points, and its unique invariant
probability, denoted by $\eta_j$, has entropy $c_j$.  Since $N_j\to\infty$,
one has $c_j\to0$.

Thus, with the fixed compact subset
$D\subset\operatorname{int}\widehat D$, all the hypotheses of
\cref{lem:small-localized-implantation} are satisfied by
$(\widehat q_j,X_j)$.  After
discarding finitely many further terms and reindexing, apply that lemma with
$r_0=d-3$.  We obtain a fixed
$(d-1)$-disk $D^\sharp$, a fixed central-slice embedding
$\kappa:\widehat D\to D^\sharp$, maps
$H_j\to\id_{D^\sharp\times S^1}$, and isolated minimal Cantor sets
$\widehat K_j=\kappa(X_j)\times\{\pi/2\}$.  Since the dynamics on
$\widehat K_j$ is conjugate to $\widehat q_j|_{X_j}$, the conclusions of
\cref{prop:localized-isolation} give
\[
 \begin{aligned}
 \Me(H_j)
   &=\{\delta_z:z\in\operatorname{Fix}(H_j)\}\cup\{\widehat\nu_j\},\\
 \He(H_j)&=\{0,c_j\},\qquad \htop(H_j)=c_j,
 \end{aligned}
\]
where $\widehat\nu_j$ is the unique invariant probability supported on
$\widehat K_j$.  More explicitly, for
\[
 \widehat\kappa_j:X_j\longrightarrow\widehat K_j,
 \qquad
 \widehat\kappa_j(x)=\bigl(\kappa(x),\pi/2\bigr),
\]
one has
\[
 \widehat\nu_j=(\widehat\kappa_j)_*\eta_j,
 \qquad
 h_{\widehat\nu_j}(H_j)=h_{\eta_j}(\widehat q_j)=c_j.
\]

As in the proof of \cref{thm:main}, a coordinate $d$-ball in $M$ contains a
tubular neighborhood of a standard circle.  Fix a smooth embedding
$\jmath:D^\sharp\times S^1\to M$, conjugate $H_j$ by $\jmath$ on the fixed
solid torus $N=\jmath(D^\sharp\times S^1)$, and extend it by the identity on
$M\setminus N$.  Denote the resulting diffeomorphism by $h_j$, and put
$K_j=\jmath(\widehat K_j)$ and $\nu_j=\jmath_*\widehat\nu_j$.  Since every
$H_j$ is the identity on one fixed neighborhood of
$\partial D^\sharp\times S^1$, the extension is smooth across
$\partial N=\jmath(\partial D^\sharp\times S^1)$, and $h_j$ is the identity on
$M\setminus\operatorname{int}N$.  Moreover, since $\jmath$ is fixed, the
convergence
$H_j\to\id_{D^\sharp\times S^1}$ in $C^\infty$ implies
$h_j\to\id_M$ in $C^\infty$.
Conjugacy by $\jmath$ also shows that
$K_j\subset\operatorname{int}N$ is isolated, minimal, and $h_j$-invariant,
and that $\nu_j$ is its unique invariant probability.

By construction, $h_j(N)=N=h_j^{-1}(N)$.  If $\mu$ is an ergodic
$h_j$-invariant probability, then $\mu(N)\in\{0,1\}$.  In the zero-mass case,
$h_j=\id$ $\mu$-almost everywhere, so ergodicity forces
$\mu=\delta_z$ for a fixed point $z$.  In the full-mass case, conjugacy by
$\jmath$ reduces the measure to the local classification of $H_j$.
Consequently, entropy invariance under conjugacy and the variational principle
give
\[
 \begin{aligned}
 \Me(h_j)&=\{\delta_z:z\in\operatorname{Fix}(h_j)\}\cup\{\nu_j\},
 & h_{\nu_j}(h_j)&=c_j,\\
 \He(h_j)&=\{0,c_j\},
 & \htop(h_j)&=c_j.
 \end{aligned}
\]

Finally, \cref{lem:small-localized-implantation} supplies a smooth
relative-boundary isotopy from $\id_{D^\sharp\times S^1}$ to $H_j$.
Write it as $(H_{j,\lambda})_{0\leq\lambda\leq1}$ and define
\[
 h_{j,\lambda}(x)=
 \begin{cases}
  \jmath\circ H_{j,\lambda}\circ\jmath^{-1}(x),&x\in N,\\
  x,&x\in M\setminus N.
 \end{cases}
\]
All $H_{j,\lambda}$ are the identity on the same boundary collar, so this
piecewise family is jointly smooth in $(\lambda,x)$ and has endpoints
$\id_M$ and $h_j$.
Therefore $h_j\in\operatorname{Diff}_0^\infty(M)$.

Since $\id_M$ has zero topological
entropy, it satisfies \eqref{IE} vacuously, whereas every $h_j$ omits all
ergodic entropy values in $(0,c_j)$.  Thus the class satisfying \eqref{IE} is
not $C^\infty$ open.
\end{proof}

\medskip
\noindent\textbf{Acknowledgments.}
W. Lin is supported by the National Natural Science Foundation of China (No.~124B2010).  X. Tian is supported by the National Natural Science Foundation of China (No.~12471182).

\end{document}